\documentclass[a4paper,11pt]{article}
\usepackage[utf8]{inputenc}
\usepackage{amsmath,amsthm,verbatim,amssymb,amsfonts,amscd, graphicx}
\usepackage{hyperref}
\usepackage{thmtools}
\usepackage[capitalise, noabbrev]{cleveref}
\usepackage{tikz}
\usepackage[mode=buildnew]{standalone}
\usepackage{authblk}
\usepackage[a4paper, margin=2cm]{geometry}
\usepackage{array}
\usepackage{booktabs}
\usepackage{makecell}
\usepackage{enumitem}
\usepackage{tikz-cd}
\usepackage{subcaption}
\usepackage{float}
\usepackage{color, colortbl}
\usepackage{diagbox}
\usepackage{algorithm}%
\usepackage{algorithmicx}%
\usepackage{algpseudocode}%

\NewDocumentCommand{\rsubseteq}{m}
 {%
  \rotatebox[origin=c]{#1}{$\subseteq$}%
 }
\DeclareMathOperator{\Cov}{Cov}
\DeclareMathOperator{\Cr}{Core}
\DeclareMathOperator{\cC}{Core\check{C}ech}
\DeclareMathOperator{\Core}{\mathit{d}}
\DeclareMathOperator{\aCr}{DelCore}
\DeclareMathOperator{\Vor}{Vor}

\newcounter{dummy} \numberwithin{dummy}{section}
\newtheorem{theorem}[dummy]{Theorem}
\newtheorem{corollary}[dummy]{Corollary}
\newtheorem{lemma}[dummy]{Lemma}
\newtheorem{definition}[dummy]{Definition}

\theoremstyle{remark}

\newtheorem{example}[dummy]{Example}

\makeatletter
\def\keywords{\xdef\@thefnmark{}\@footnotetext}
\makeatother

\makeatletter
\renewcommand*{\@cite@ofmt}{\bfseries\hbox}
\makeatother

\title{Core Bifiltration}

\author[1]{Nello Blaser}
\author[2]{Morten Brun}
\author[1]{Odin Hoff Gardaa}
\author[1,2]{Lars M. Salbu}

\affil[1]{Department of Informatics, University of Bergen, Norway}
\affil[2]{Department of Mathematics, University of Bergen, Norway}
\affil[ ]{\textit{\{nello.blaser,morten.brun,odin.garda,lars.salbu\}@uib.no}}

\date{}

\begin{document}

\maketitle

\begin{abstract}
The motivation of this paper is to recognize a geometric shape from a noisy sample in the form of a point cloud. Inspired by the HDBSCAN clustering algorithm, we introduce the core dissimilarity, from which we construct the core bifiltration. We also consider the Delaunay core bifiltration by intersecting with Voronoi cells, giving us a filtered simplicial complex of smaller size. A major advantage of the (Delaunay) core bifiltration is that, for each filtration value, it admits a good cover of balls. By the persistent nerve theorem, the nerve of this cover is homotopy equivalent to the (Delaunay) core bifiltration. We show that the multicover-, core- and Delaunay core bifiltrations are all interleaved, and that they enjoy similar stability properties with respect to the Prohorov distance. We have performed experiments with the Delaunay core bifiltration. In the experiments, we calculated persistent homology along lines in the two-dimensional persistence parameter space, as well as multipersistence module approximations and Hilbert functions for the full Delaunay core bifiltration.
\end{abstract}



\section{Introduction}
Multiparameter persistent homology has recently emerged as an important part of topological data analysis \cite{botnan2023introduction}. In multiparameter persistent homology, we study the changes in homology of a topological space equipped with a multiparameter filtration as these parameters vary. In practice, many filtrations have two parameters, so-called bifiltrations. Examples include the degree-Rips bifiltration \cite{lesnick2015interactive,blumberg2022stability}, the density-Rips bifiltration \cite{botnan2023introduction,MR2469164}, the sublevel Delaunay-\v{C}ech bifiltration \cite{alonso2024delaunay}
and the multicover bifiltration \cite{sheehy12multicover}. The study of different bifiltrations and their relationships with each other is an active research area.

In this paper, we introduce the core bifiltration using the core distance from \cite{hdbscan}. For a finite point cloud $A$ in a metric space $(M,d)$ and $k>0$, the $k$-core distance $\Core^A_k(x)$ of a point $x$ in $M$ is the radius of the smallest closed metric ball centered at $x$ that contains at least $k$ points of $A$. We denote the core bifiltration by $\Cr^\beta(A)$ where $\beta>0$ is a fixed real parameter. For filtration values $r,k>0$, the set $\Cr_{r,k}^\beta(A)$ consists of the points in $M$ that are within distance $r$ to some point $a$ in $A$ satisfying $\Core^A_k(a)\leq r/\beta$. Thus, the parameter $\beta$ balances the relative weights of the core- and the metric distance. This is analogous to the $\alpha$-parameter in the hierarchical clustering algorithms of \cite{chaudhuri2010_cluster_tree} and \cite{rolle2024stable}.

The core bifiltration generalizes the degree-\v{C}ech and the degree-Rips bifiltrations in the sense that both can be recovered by setting $\beta=1/2$ and applying the appropriate nerve constructions.


In order to mitigate the fact that the core bifiltration does not scale well with the number of data points, we introduce the Delaunay core bifiltration $\aCr^\beta(A)$ of a point cloud $A\subseteq\mathbb{R}^n$. As discussed in \cref{sec:alphacore}, the Delaunay complex has $\mathcal{O}(|A|^{\lceil\frac{n}{2}\rceil})$ simplices resulting in a size of $\mathcal{O}(|A|^{\lceil\frac{n}{2}\rceil+1})$ for the Delaunay core bifiltration. This is similar to the size $\mathcal{O}(|A|^{\lceil\frac{n+1}{2}\rceil})$ of the sublevel Delaunay-\v{C}ech bifiltration, which is constructed from a real-valued function on the input point cloud $A$. This function is often a density function, requiring the selection of a bandwidth parameter. In contrast, our construction is parameter-free.

The core- and the Delaunay core bifiltrations are both interleaved with the multicover bifiltration. This interleaving allows us to transfer stability results from \cite{blumberg2022stability}, in a weakened form, from the multicover bifiltration to the (Delaunay) core bifiltration. 

Similar to the degree-Rips bifiltration, the core bifiltration, and in particular the Delaunay core bifiltration, is computationally more accessible than the multicover bifiltration because of its smaller size. There are ongoing efforts in the research community aiming to improve the feasibility of computing the multicover bifiltration through techniques such as, for example, sparsification \cite{buchet2023sparse} and exploring related combinatorial bifiltrations \cite{Corbet2023,Edelsbrunner2021,Edelsbrunner2023}.

\textbf{Contributions.} Our contributions are as follows. 
\begin{enumerate}
    \item We introduce two bifiltrations of spaces, namely the \textit{core bifiltration} and the \textit{Delaunay core bifiltration}. These are inspired by the density-dependent clustering algorithm HDBSCAN \cite{hdbscan}.
    \item In theorems \ref{corecovtheorem} and \ref{alphacoremulticov} we
    describe interleavings between the (Delaunay) core and multicover bifiltrations summarized in the following diagram of inclusions:
    
\resizebox{0.95\linewidth}{!}{
\begin{minipage}{\linewidth}
\begin{align*}
\aCr^\beta_{r,k}(A) \subseteq
\begin{matrix}
  &&  \Cov_{(1+\beta^{-1})r, k}(A) \\
  & \rsubseteq{45} && \rsubseteq{-45} \\
  \Cr^\beta_{r,k}(A) &&&& \aCr^\beta_{\max\{1,2\beta\}(1+\beta^{-1})r, k}(A), \\
  & \rsubseteq{-45} && \rsubseteq{45} \\
  && \aCr^\beta_{(2\beta+1)r, k}(A)
  \end{matrix}%
\end{align*}
\end{minipage}
}
\vspace{1em}\\
where $\Cov_{r, k}(A)\subseteq\mathbb{R}^n$ denotes the multicover bifiltration of $A$ for filtration values $r,k>0$. The parameter $r$ corresponds to the radius parameter in the \v{C}ech complex, and $k$ is a density parameter. See \cref{fig:filtered_nerves_illustrations} for a toy example where the core- and multicover bifiltrations differ.
\item In each filtration degree, the core bifiltration admits a cover consisting of metric balls. Applying the standard nerve lemma, we get that the core bifiltration is homotopy equivalent to the geometric realization of the core \v{C}ech bifiltration $\cC^\beta(A)$. A similar argument also applies to the Delaunay core bifiltration. 
\item Using an approach similar to \cite{blumberg2022stability} we get Prohorov stability results for the core- and Delaunay core bifiltrations.
\item We have performed experiments on synthetic point clouds with noise. For each point cloud, we compute the Delaunay core persistent homology along a line\footnote{The implementation for computing (Delaunay) core persistence along a line slice of the parameter space is available at \url{https://github.com/odinhg/core}.} and calculate the bottleneck distance to a noise-free ground truth. The experiments suggest that the Delaunay core bifiltration might be suitable for analyzing point clouds when noise robustness is desired. Additionally, we have implemented the full Delaunay core bifiltration which is now part of the \texttt{multipers}\footnote{The \texttt{multipers} library is available at \url{https://github.com/DavidLapous/multipers}.} library \cite{multipers}, making it easily accessible for researchers and practitioners. We also include experiments comparing the full Delaunay core bifiltration with a sublevel Delaunay-\v{C}ech bifiltration of a codensity function \cite{alonso2024delaunay}, and with the multicover bifiltration computed using the rhomboid tiling bifiltration \cite{Corbet2023}.
\end{enumerate}

\begin{figure}[h]
    \centering
    \includegraphics[width=0.9\textwidth]{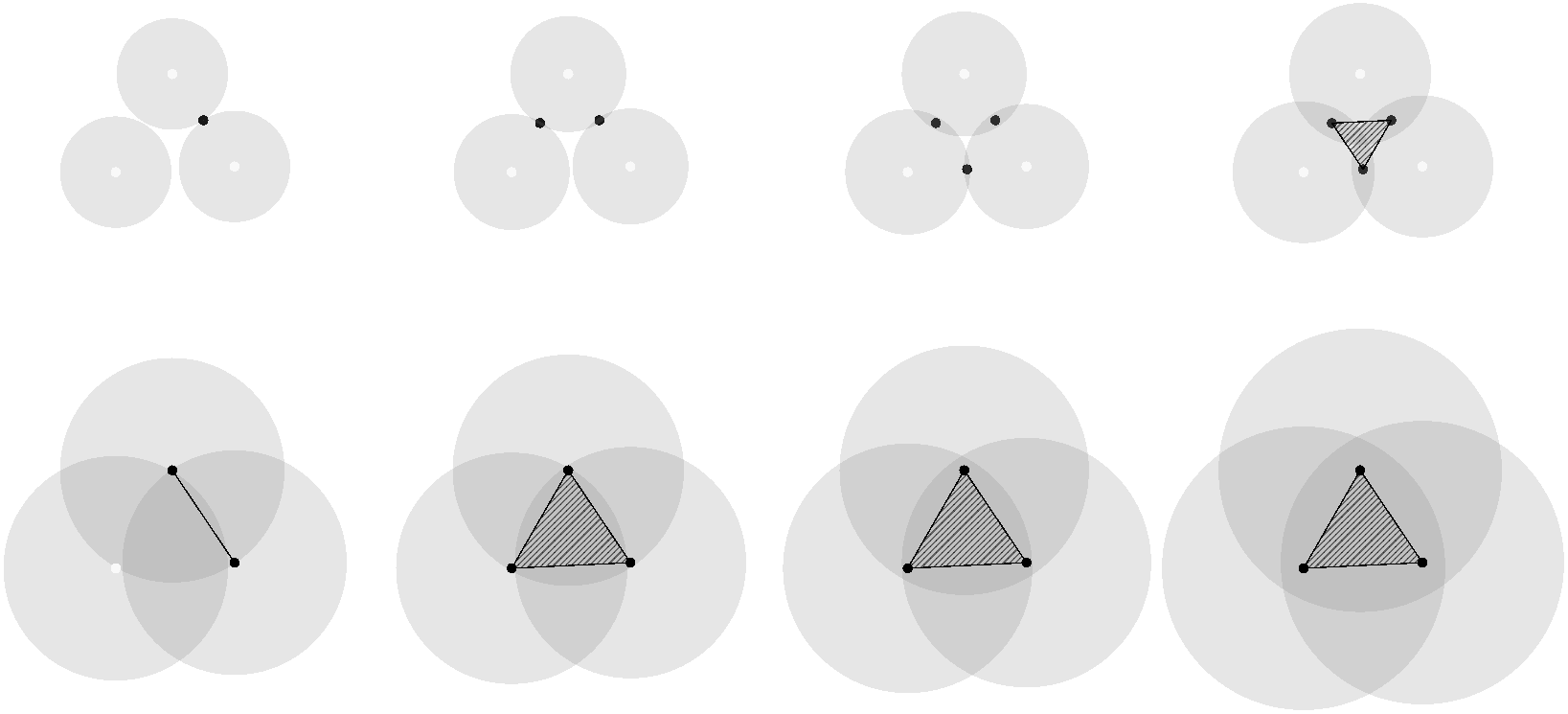}

    \caption{For a fixed $k$, the multicover bifiltration $\Cov_{r,k}(A)$ admits a covering consisting of $k$-intersections of balls of radius $r$. The figure compares the nerve of this covering (top row) and the core \v{C}ech bifiltration (bottom row) for a point cloud $A=\{(0,0), (1, \sqrt{3}), (2.1, 0.1)\}$ in $\mathbb{R}^2$ with $k=2$ fixed. The scale parameter $r$ increases from left to right. Light gray shading represents metric balls centered at points in $A$, while black points, lines, and shaded triangles indicate the simplices in the nerves. In the multicover nerve, the vertices correspond to $k$-intersections of the metric balls, whereas the core \v{C}ech bifiltration has vertices equal to the original point cloud $A$. A scale of $2r$ is used for the core bifiltration to ensure visual alignment between the rows.}
  \label{fig:filtered_nerves_illustrations}
\end{figure}


This manuscript is structured as follows: In Section~\ref{sec:background}, we review some notions relevant to multiparameter persistent homology, including a formal definition of the multicover bifiltration. \cref{sec:core} defines the core dissimilarity and its corresponding core bifiltration and shows the interleaving to the multicover bifiltration. In \cref{sec:alphacore}, we introduce the Delaunay core bifiltration and show how it is related to the core bifiltration and the multicover bifiltration. Stability results are shown in \cref{sec:stability}. In \cref{sec:experiments}, we showcase properties of the Delaunay core bifiltration in some experiments on noisy point cloud datasets. Finally, we conclude the paper in \cref{sec:conclusion}.

\section{Background}\label{sec:background}

Let $P$ be a partially ordered set, or \textbf{poset}. A \textbf{filtration} (of sets) over $P$ is a collection of sets $C=\{C_{p}\}_{p\in P}$ where $C_{p}\subseteq C_{p'}$ whenever $p\leq p'$. For a poset $P$, we let $P^{op}$ denote the \textbf{opposite poset} with the same underlying set $P$ but with the order reversed, i.e., $p\leq q$ in $P^{op}$ if $q\leq p$ in $P$. For posets $P$ and $Q$, we can form the \textbf{product poset} $P\times Q$, where $(p,q)\leq(p',q')$ if both $p\leq p'$ and $q\leq q'$.

A \textbf{simplicial complex} is a pair $(K,V)$ where $V$ is a set (called the \textbf{vertex set}) and $K$ is a set of finite subsets of $V$ (called \textbf{simplices}) such that if $\sigma$ is a simplex in $K$ and $\tau\subseteq\sigma$ then $\tau$ is also a simplex in $K$. A \textbf{filtered simplicial complex} (over $P$) is a collection of simplicial complexes $\{(K_p,V)\}_{p\in P}$ such that $\{K_p\}_{p\in P}$ is a filtration of sets over $P$. We let $(0,\infty)$ denote the set of positive real numbers with the standard total ordering, and $[0,\infty]$ is the extended non-negative real line. 

\begin{definition}[Multicover Bifiltration]\label{def:multicover bifiltration}
    Let $(M,d)$ be a metric space, and $A\subseteq M$ a finite subspace. The \textbf{multicover bifiltration} $\Cov(A)$ on $A$ is the filtration over $(0,\infty)\times(0,\infty)^{op}$ whose sets for (real) parameters $r,k > 0$ are given as 
    \begin{equation} \label{eq:cov}
        \Cov_{r,k}(A)=\{x\in M\,|\,d(a,x)\leq r \textrm{ for at least $k$ points }a\in A\}.
    \end{equation} 
\end{definition}
We use closed balls instead of open balls like they do in \cite[Def.\ 2.7]{blumberg2022stability}. However, as pointed out in \cite[Remark 2.9]{blumberg2022stability}, this does not matter for stability results since the two versions are $0$-interleaved. Note that replacing $k$ by its ceiling $\lceil k \rceil$ in \eqref{eq:cov} so that it reads ``for at least $\lceil k \rceil$ points'', gives an equivalent definition. When we later discuss interleavings, it is convenient not to restrict $k$ to integer values.

\begin{definition}[Dissimilarity]
    A \textbf{dissimilarity (over $[0,\infty]$)} is a function $F:S\times T\to [0,\infty]$ where $S$ and $T$ are sets. 
\end{definition}
An example of a dissimilarity is the (extended) metric $d$ of a metric space $(M,d)$. 
\begin{definition}[Balls]
    For a dissimilarity $F:S\times T\to [0,\infty]$, the \textbf{(closed) $F$-ball} around $s\in S$ with radius $r\geq 0$ is the set
    \begin{equation*}
        B_F(s,r) = \{t\in T\,|\, F(s,t)\leq r\}.
    \end{equation*}
\end{definition}

We can construct a filtered simplicial complex from a dissimilarity as follows: 
\begin{definition}[Dowker Nerve]
    The \textbf{Dowker nerve} $DF=\{(DF_r,S)\}_{r\in(0,\infty)}$ of a dissimilarity $F:S\times T\to [0,\infty]$ is the filtered simplicial complex where
\begin{equation*}
    DF_r = \left\{\sigma\subseteq S \textrm{ finite} \,|\, \textrm{there exists } t\in T \textrm{ such that } F(s,t) \leq r \textrm{ for all } s\in\sigma\right\}.
\end{equation*}
\end{definition}


For a topological space $X$, a \textbf{cover} is a collection of subsets $\mathcal{U}=\{U_i\}_{i\in I}$ of $X$ such that $\bigcup_i U_i = X$. The cover $\mathcal{U}$ is said to be \textbf{good} if the intersection $\bigcap_{j\in J} U_j$ is either empty or contractible for every finite subset $J\subseteq I$.  The cover is \textbf{closed} it consists of closed sets. If $X\subseteq\mathbb{R}^n$ is a Euclidean subspace, we say that the cover $\mathcal{U}$ is \textbf{convex} if it consists of convex sets. Note that every convex cover is good.
\begin{definition}[Nerve of a Cover]
    Let $\mathcal{U}=\{U_i\}_{i\in I}$ be a cover of a space $X$. The \textbf{nerve} of $\mathcal{U}$ is the simplicial complex $(N\mathcal{U},I)$ where $N\mathcal{U}$ consists of the finite subsets $J$ of $I$ with the property that the intersection $\bigcap_{j \in J} U_j$ is nonempty.
\end{definition}

\begin{lemma}[Functorial Nerve Lemma {\cite[Thm.\ 5.9 and Thm.\ 3.9]{Bauer_2023}}]\label{lemma:nerve}
    Let $\mathcal{U}=\{U_i\}_{i\in I}$ be a finite closed and convex cover of a Euclidean subspace $X\subseteq \mathbb{R}^n$. There exists a homotopy equivalence $\rho_X:X\to |N\mathcal{U}|$ from $X$ to the geometric realization $|N\mathcal{U}|$ of the nerve $N\mathcal{U}$. Furthermore, if $Y\subseteq X$ is a subspace with a finite closed and convex cover $\mathcal{V}=\{V_i\}_{i\in I}$ such that $V_i\subseteq U_i$ for all $i\in I$, then the following diagram commutes up to homotopy:
    \begin{equation}\label{nervediagram}
        \begin{tikzcd}
            Y\arrow[d,hook]\arrow[r,"\rho_Y"] & \vert N\mathcal{V} \vert\arrow[d,hook]\\
            X\arrow[r,"\rho_X"]& \vert N\mathcal{U} \vert.
        \end{tikzcd}
    \end{equation}
\end{lemma}
We observe that taking homology of the diagram \eqref{nervediagram} gives a diagram that commutes strictly, where the two horizontal maps are isomorphisms.

\section{Core Bifiltration}\label{sec:core}
In this section, we introduce the core bifiltration, using a specific family of dissimilarities depending on a fixed real parameter $\beta>0$. This parameter plays a similar role to the $\alpha$ used in the robust single-linkage clustering algorithm of \cite{chaudhuri2010_cluster_tree}, decoupling the metric from the $k$-nearest neighbor distance by a multiplicative constant. We also examine the Dowker nerves of these dissimilarities. The core bifiltration is a filtered space that is interleaved with the multicover bifiltration.



\begin{definition}[Core Distance {\cite[Def.\ 5]{hdbscan}}]
    Let $A\subseteq(M,d)$ be a finite metric subspace, let $k > 0$ and let $x\in M$. The \textbf{$k$-core distance} $\Core^A_k(x)$ of $x$ to $A$ is the distance from $x$ to one of its $\lceil k\rceil$-th nearest neighbors in $A$. We use the convention that, if $k>|A|$, then the $k$-core distance is infinite. 
\end{definition}

In particular, we note that if $a\in A$ and $0<k\leq 1$, then the $k$-core distance $\Core^A_k(a)$ is zero since $a$ is its own nearest neighbor. Furthermore, it can be helpful to keep in mind the following three equivalent conditions related to the core distance as we will use them in later proofs: Let $A\subseteq(M,d)$ be a finite metric subspace, then for all $x\in M$ and $r,k>0$, we have $\Core^A_k(x) \leq r$ if and only if there exist $a_1,\ldots, a_{\lceil k\rceil}\in A$ such that $d(a_i,x) \leq r$ for all $a_i$ if and only if $\vert B_d(x,r)\cap A\vert\geq k$.


For HDBSCAN, the \textbf{mutual reachability distance} $G_k:A\times A\to [0,\infty]$, where 
\begin{equation}\label{mutualreach}
    G_k(a,a')=\max\left\{\Core^A_k(a),\Core^A_k(a'),d(a,a')\right\},
\end{equation}
is used in one of the main steps in the clustering algorithm {\cite[Def.\ 7 and Alg.\ 1]{hdbscan}}. We look at directed versions of this distance equipped with a scale parameter $\beta>0$.
\begin{definition}[Core Dissimilarity]\label{coredissdef}
    Let $A\subseteq(M,d)$ be a finite metric subspace, fix $\beta>0$ and let $k>0$. The \textbf{$k$-core dissimilarity} of $A$ in $M$ is the dissimilarity
    $\Lambda^\beta_k\colon A\times M\to [0,\infty]$ given by
    \begin{equation*}
    \Lambda^\beta_k(a,x)=\max\left\{\beta\Core^A_k(a),d(a,x)\right\}.
    \end{equation*}
\end{definition}

Following our convention, we denote balls of radius $r$ centered at $a$ with respect to the core dissimilarity by
\begin{equation}\label{lambdaballs}
    B^\beta_{r,k}(a):=B_{\Lambda^\beta_k}(a,r)=\{x \in M \,|\, \Lambda^\beta_k(a,x) \leq r\}.
\end{equation}

Taking the union of such balls over $A$, we get a bifiltration.
\begin{definition}[Core Bifiltration]\label{corebifdef}
    Let $(M,d)$ be a metric space, $A\subseteq M$ a finite subset and fix $\beta>0$. The \textbf{core bifiltration} $\Cr^\beta(A)$ on $A$ is a filtration over $(0,\infty)\times(0,\infty)^{op}$ whose sets $\Cr^\beta_{r,k}(A)$ for (real) parameters $r,k>0$ are given as the union of all $\Lambda^\beta_k$-balls of radius $r$, i.e.,
\begin{equation*}
    \Cr^\beta_{r,k}(A)= \bigcup_{a\in A}B^\beta_{r,k}(a).
\end{equation*}
\end{definition}

Although we fix the parameter $\beta$ in this paper, allowing $\beta$ to vary yields a trifiltration, i.e., a filtration over $(0,\infty)\times(0,\infty)^{op}\times(0,\infty)^{op}$ since $\Lambda^\beta_k(a,x)\leq\Lambda^{\beta'}_k(a,x)$ whenever $\beta\leq\beta'$. The \textbf{$r$-thickening} of a metric subspace $A\subseteq(M,d)$ is the union of closed metric balls of radius $r\geq0$ centered at the points in $A$. Geometrically, $\Cr^\beta_{r,k}(A)$ is then the $r$-thickening of the points in $A$ having at least $k$ neighbors (including itself) in $A$ within distance $r/\beta$ (see \cref{fig:union_of_balls}). In particular, $\Cr^\beta_{r,k}(A)$ is covered by metric balls, and in the case where $M=\mathbb{R}^n$ with the Euclidean metric $d = d_E$, the intersections of such balls are either empty or contractible. In particular, the collection of balls $\mathcal{B}^\beta_{r,k}=\{B^\beta_{r,k}(a)\}_{a\in A}$ forms a closed and convex cover of $\Cr^\beta_{r,k}(A)$.

\begin{figure}[h]
    \centering
    \includegraphics[width=0.6\textwidth]{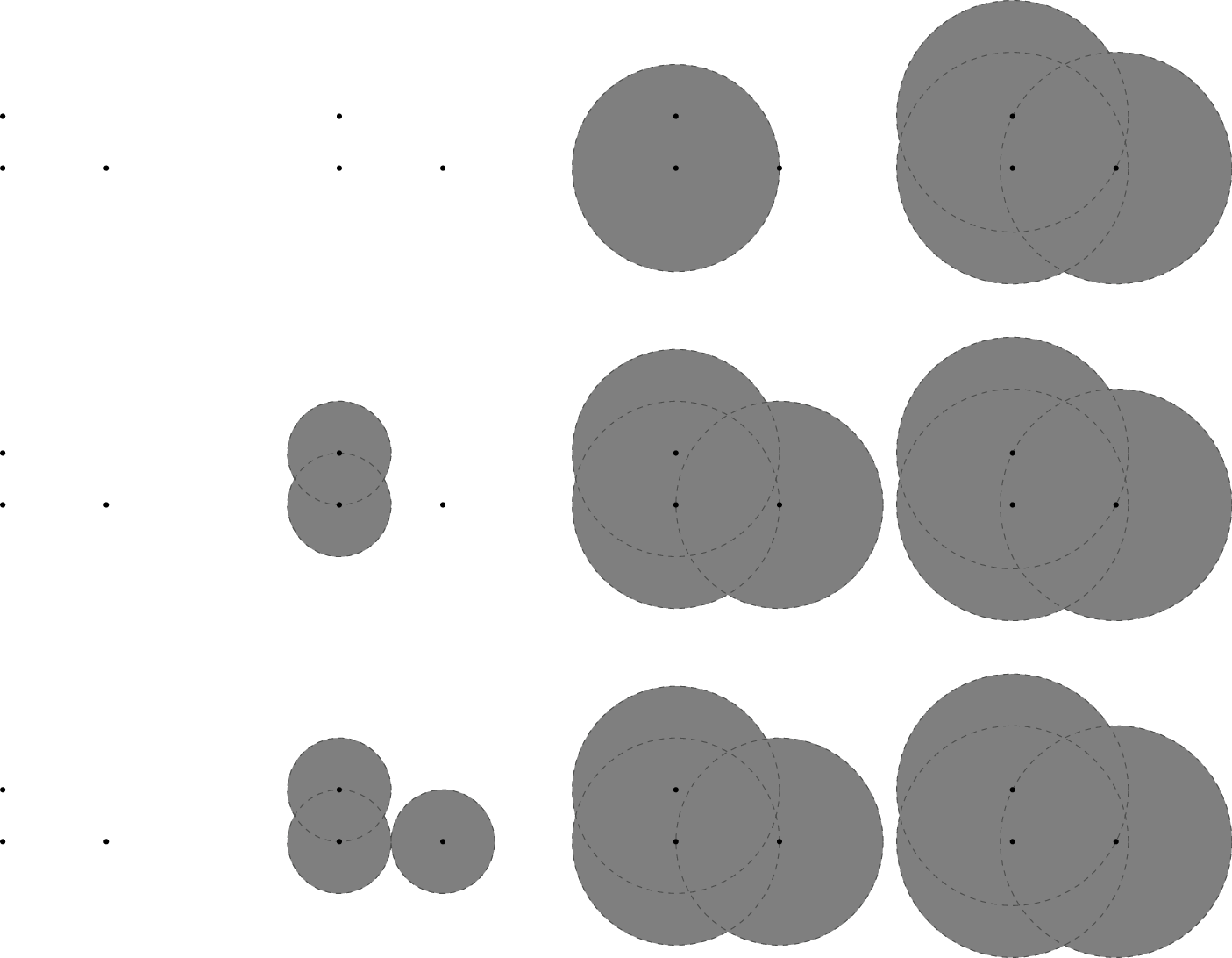}
    \caption{The core bifiltration $\Cr^1(A)$ with the union of $\Lambda_{k,1}$-balls shown in gray for different radii (increasing from left to right) and values of $k$ where $A=\{(0,0), (0, 1), (2, 0)\}\subseteq\mathbb{R}^2$. The top, middle and bottom rows correspond to $k=3$, $k=2$ and $k=1$, respectively. The point cloud $A$ and the boundaries of the metric balls are shown for clarity.}
    \label{fig:union_of_balls}
\end{figure}

\begin{definition}[Core \v{C}ech Bifiltration]
    Let $A\subseteq(M,d)$ finite metric subspace and fix $\beta>0$. The \textbf{core \v{C}ech bifiltration} $\cC^\beta(A)=\{(\cC^\beta_{r,k}(A),A)\}_{r,k>0}$ on $A$ is the filtered simplicial complex over $(0,\infty)\times(0,\infty)^{op}$ where $\cC^\beta_{r,k}(A)=D(\Lambda^\beta_k)_r=N\mathcal{B}^\beta_{r,k}$.
\end{definition}
Let us record the above discussion in a lemma, using the Nerve Lemma (\cref{lemma:nerve}).

\begin{lemma}\label{lemma:corecechequiv}
    Let $A\subseteq\mathbb{R}^n$ be a finite Euclidean subspace and $\beta>0$. For each $r,k>0$, the space $\Cr^\beta_{r,k}(A)\subseteq \mathbb{R}^n$ is homotopy equivalent to the geometric realization of the simplicial complex $\cC^\beta_{r,k}(A)$. 
\end{lemma}
Note that the functorial part of the Nerve Lemma \ref{lemma:nerve} gives that these degreewise homotopy equivalences induce isomorphisms between the persistent homology groups of $\Cr^\beta(A)$ and $\cC^\beta(A)$.


To connect the core bifiltration with the multicover bifiltration, we consider a second directed version of the mutual reachability distance \eqref{mutualreach}, namely the dissimilarity $\Gamma_k:A\times M\to [0,\infty]$ where 
\begin{equation}\label{def:gamma}
\Gamma_k(a,x)=\max\left\{\Core^A_k(x),d(a,x)\right\}.
\end{equation}
This dissimilarity is slightly less well-behaved, compared to $\Lambda^\beta_k$, as the balls $B_{\Gamma_k}(a,r)$ no longer have to be contractible even in the Euclidean case (see \cref{fig:gamma_balls}). However, it is of interest for the following reason:

\begin{lemma}\label{covlemma}
    Let $A\subseteq(M,d)$ be a finite subset and let $r,k>0$. The union of balls $\bigcup_{a\in A} B_{\Gamma_{k}}(a,r)$ is exactly the set $\operatorname{Cov}_{r,k}(A)$ from the multicover bifiltration.
\end{lemma}
\begin{proof}
    Observe that asking for $x\in M$ to satisfy $\Core^A_k(x) \leq r$ is equivalent to the existence of at least $\lceil k\rceil$ points $a_i$ in $A$ satisfying $d(a_i,x) \leq r$. Now, if $x\in B_{\Gamma_k}(a,r)$ for some $a\in A$, we have that $\Core^A_k(x) \leq r$ so it follows from our observation that $x\in\Cov_{r,k}(A)$. Conversely, if $x\in\Cov_{r,k}(A)$ then again by the above observation we have that $\Core^A_k(x) \leq r$, and $d(a,x) \leq r$ for at least one $a\in A$ since $\lceil k\rceil\geq 1$.
\end{proof}

\begin{figure}[h]
    \centering
    \includegraphics[width=\textwidth]{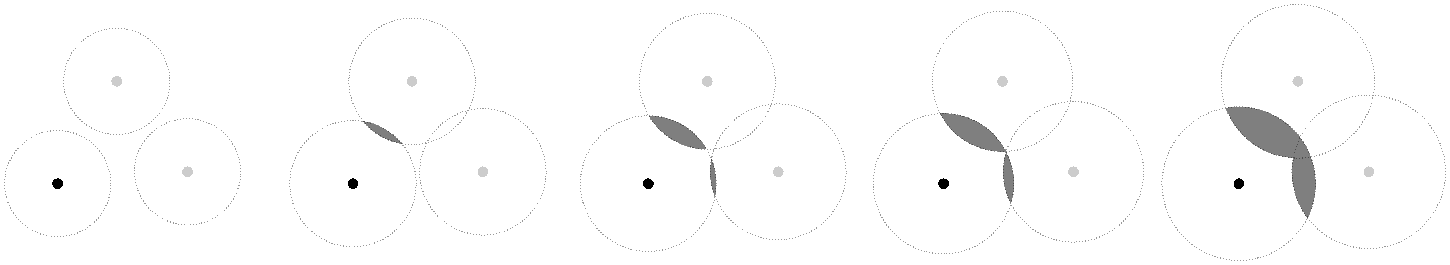}
    \caption{The $\Gamma_2$-ball $B_{\Gamma_2}((0,0), r)$ shown in gray as $r$ increases from left to right, with $A=\{(0,0), (1, \sqrt{3}), (2.1, 0.1)\}$.}
    \label{fig:gamma_balls}
\end{figure}

So far, we have introduced the core dissimilarities $\Lambda^\beta_k$ from which we defined the core bifiltration $\Cr^\beta(A)$ (by taking unions of closed balls), and the core \v{C}ech bifiltration $\cC^\beta(A)$ (by taking the Dowker nerves). They are degreewise homotopy equivalent in the Euclidean case (Lemma \ref{lemma:corecechequiv}). We now present an interleaving between the multicover bifiltration $\Cov(A)$ and the core bifiltration $\Cr^\beta(A)$. 
\begin{theorem}\label{corecovtheorem}
    Let $(M,d)$ be a metric space with a finite subspace $A\subseteq M$ and let $\beta>0$. The core bifiltration $\Cr^\beta(A)$ and the multicover bifiltration $\Cov(A)$ are interleaved as follows:
    $$
    i)\,\Cr^\beta_{r, k}(A)
    \subseteq
    \Cov_{(1+\beta^{-1})r, k}(A)
    \quad
    \textrm{and}
    \quad
    ii)\,\Cov_{r, k}(A)    
    \subseteq
    \Cr^\beta_{\max\{1,2\beta\}r,k}(A).
    $$
\end{theorem}
\begin{proof}    
    \textit{i)} If $x$ is a point in $\Cr^\beta_{r,k}(A)$, then there exists some $a$ in $A$ with $\Core^A_k(a) \leq r/\beta$ and $d(a,x) \leq r$. 
    We can pick ${\lceil k\rceil}$ points $a_1, \dots, a_{\lceil k\rceil}$ in $A$ whose distance to $a$ is less than or equal to $r/\beta$.
    For these points, we have $d(a_i,x)\leq d(a_i,a)+d(a,x) \leq r/\beta + r$, so the $k$-core distance $\Core^A_k(x)$ is less than or equal to $(1+\beta^{-1})r$ and $x$ is in $\Cov_{(1+\beta^{-1})r,k}(A)$.

    \textit{ii)} If $x$ is a point in $\Cov_{r,k}(A)$, then by \cref{covlemma} it is in some $\Gamma_k$-ball $B_{\Gamma_{k}}(a,r)$. In particular, $d(a,x) \leq r$ and $\Core^A_k(x) \leq r$.
    This means that there are at least $\lceil k\rceil$ points $a_1,a_2, \dots, a_{\lceil k\rceil}$ in 
    $A$ whose distance to $x$ is less than or equal to $r$.
    Now, $d(a_i,a)\leq d(a_i,x)+d(a,x) \leq 2r$, so the ${\lceil k\rceil}$ points $a_1, \dots, a_{\lceil k\rceil}$ have distance to $a$ less than or equal to $2r$. Thus, $\beta\Core^A_k(a) \leq 2\beta r$, and the point $x$ is in $B^\beta_{r',k}(a)\subseteq\Cr^\beta_{r',k}(A)$ where $r'=\max\{1,2\beta\}r$.
\end{proof}

It is worth noting that the interleaving factors in \cref{corecovtheorem} are minimized in terms of their product by setting $\beta=1/2$. In this case, the product of the two interleaving factors is 3. This also applies to \cref{alphacoremulticov}. The following example relates the core bifiltration to the degree-\v{C}ech bifiltration and the well-known degree-Rips bifiltration.

\begin{example}
Let $A\subset(M,d)$ be a finite metric subspace and let $\check{C}_r(A)$ denote the usual \v{C}ech complex on $A$ in filtration degree $r$. The degree-\v{C}ech bifiltration \cite{alonso2024probabilistic, blumberg2022stability} $\mathcal{D}\check{C}(A)$ is defined, in filtration degree $(r,k)$, as the maximal subcomplex of $\check{C}_r(A)$ whose vertices have degree at least $k-1$ in the $1$-skeleton of $\check{C}_r(A)$. The degree-Rips bifiltration is constructed similarly by simply replacing the \v{C}ech filtration with the usual Rips filtration. We recover the degree-\v{C}ech bifiltration via the core \v{C}ech bifiltration by setting $\beta=1/2$, i.e.,
$$
\mathcal{D}\check{C}_{r,k}(A)=\cC^{1/2}_{r,k}(A).
$$
To see this, let $\sigma\subseteq A$ be finite. The existence of an $x\in M$ satisfying $d(a,x)\leq r$ for all $a\in\sigma$ is equivalent to having $\sigma\in\check{C}_r(A)$. Moreover, the condition $\Core^A_k(a)\leq2r$ is equivalent to $a$ having at least $k$ neigbors in $A$ within a radius of $2r$, or in other words $a$ being of degree at least $k-1$ in the $1$-skeleton of $\check{C}_r(A)$.

It is also possible to define the Rips nerve $RF=\{(RF_r,S)\}_{r\in(0,\infty)}$ of a dissimilarity $F\colon S\times T\to [0,\infty]$. A simplex $\sigma\subseteq S$ is in $RF_r$ if all subsets of $\sigma$ with cardinality at most two are in $DF_r$. We then recover the degree-Rips bifiltration on $A$ as the Rips nerve of $\Lambda_{k,1/2}\colon A\times M\to [0,\infty]$.
\end{example}

\section{Delaunay Core Bifiltration}\label{sec:alphacore}
We now use a standard approach of intersecting with Voronoi cells, to get a smaller variation of the core bifiltration.

\begin{definition}
    Let $A\subseteq\mathbb{R}^n$ be a finite Euclidean subspace. The \textbf{Voronoi cell} $\Vor_A(a)$ of $a\in A$ is the set of all points in $\mathbb{R}^n$ that are at least as close to $a$ as to any other point in $A$, namely,
    \begin{equation*}
        \Vor_A(a)=\{x\in\mathbb{R}^n \,|\, d(a,x)\leq d(a',x) \textrm{ for all } a'\in A\}.
    \end{equation*}
\end{definition}
Each Voronoi cell is closed and convex, so the collection $\Vor(A)=\{\Vor_A(a)\}_{a\in A}$ forms a closed and convex cover of $\mathbb{R}^n$. The nerve of $\Vor(A)$ is called the \textbf{Delaunay complex} of $A$. We recall that the balls $B^\beta_{r,k}(a)$ (from \eqref{lambdaballs}) covering the core bifiltration $\Cr^\beta_{r,k}(A)$ are either empty or closed metric balls $B_d(a,r)$. Euclidean balls are convex, so the closed \textbf{Voronoi balls}
\begin{equation*}
    W^\beta_{r,k}(a):=B^\beta_{r,k}(a)\cap \Vor_A(a)
\end{equation*}
are either empty or convex, and thus they form a closed and convex cover of their union.
\begin{definition}
    Let $A\subseteq\mathbb{R}^n$ be a finite Euclidean subspace and let $\beta>0$. The \textbf{Delaunay core bifiltration} $\aCr^\beta(A)$ of $A$ is the bifiltration given by the union of Voronoi balls:
    \begin{equation*}
        \aCr^\beta_{r,k}(A)=\bigcup_{a\in A} W^\beta_{r,k}(a).
    \end{equation*}
\end{definition}
The nerve of the cover $\mathcal{W}^\beta_{r,k}=\{W^\beta_{r,k}(a)\}_{a\in A}$ is homotopy equivalent to $\aCr^\beta_{r,k}(A)$ by the Nerve Lemma (\cref{lemma:nerve}), and it is contained in the Delaunay complex $N(\Vor(A))$ whose cardinality is $\mathcal{O}(|A|^{\lceil \frac{n}{2} \rceil})$ \cite{klee_1964,Seidel1995}.

We follow \cite{lesnicknotes23}, for a notion of size relevant to the multiparameter setting. Let $\mathcal{K}=\{(K_p,V)\}_{p\in P}$ be a filtered simplicial complex over $P$, and let $\bigcup\mathcal{K}$ denote the union $\bigcup_{p\in P}K_p$. For a simplex $\sigma\in\bigcup\mathcal{K}$, define the \textbf{birth set} $b(\sigma)$ of $\sigma$ to be the smallest subset $S\subseteq P$ satisfying $\sigma\in K_p$ if and only if $p\geq s$ for some $s\in S$. Define the \textbf{size} of $\mathcal{K}$ as the sum $\sum_{\sigma\in\bigcup\mathcal{K}}|b(\sigma)|$. Since $\Core^A_k(a)=\infty$ for all $a\in A$ whenever $k>|A|$, we have that $|b(\sigma)|\leq |A|$ and thus the size of $\{N(\mathcal{W}^\beta_{r,k})\}_{(r,k)}$ is $\mathcal{O}(|A|^{\lceil \frac{n}{2} \rceil+1})$. This is comparable to the sublevel Delaunay-\v{C}ech bifiltration from \cite{alonso2024delaunay} of size $\mathcal{O}(|A|^{\lceil\frac{n+1}{2}\rceil})$. In contrast, the $m$-skeletons of the degree-Rips and core \v{C}ech bifiltrations are of size $\mathcal{O}(|A|^{m+2})$.

Unlike the standard alpha complex, which is homotopy equivalent to the \v{C}ech complex, the same is not the case for the Delaunay core and the core bifiltrations. However, they are interleaved. 

\begin{lemma}\label{alphatheorem}
    Let $A\subseteq\mathbb{R}^n$ be a finite Euclidean subspace and let $\beta>0$. The core bifiltration $\Cr^\beta(A)$ and the Delaunay core bifiltration $\aCr^\beta(A)$ are interleaved as follows:
    \begin{equation*}
        \aCr^\beta_{r,k}(A)\subseteq \Cr^\beta_{r,k}(A)\subseteq\aCr^\beta_{(2\beta+1)r,k}(A).
    \end{equation*}
\end{lemma}
\begin{proof}
It follows directly from $W^\beta_{r,k}(a)\subseteq B^\beta_{r,k}(a)$ that $\aCr^\beta_{r,k}(A)\subseteq\Cr^\beta_{r,k}(A)$. 
If $x\in\Cr^\beta_{r,k}(A)$, then there exists $a'\in A$ such that $\Core^A_k(a') \leq r/\beta$ and $d(a',x) \leq r$. Let $a\in A$ be such that $x\in \Vor_A(a)$. 
Choose $\lceil k \rceil$ points $a_1, \dots, a_{\lceil k \rceil}$ in $A$ with $d(a', a_i) \leq r/\beta$ for $i = 1, \dots, \lceil k\rceil$.
Since $d(a,x)\leq d(a',x) \leq r$, the triangle inequality implies that
\begin{equation*}
    d(a,a_i)\leq d(a,x)+d(x,a')+d(a',a_i) \leq r+r+r/\beta = \beta^{-1}(2\beta+1)r,
\end{equation*}
and $\beta\Core^A_k(a) \leq (2\beta+1)r$ (see \cref{fig:alpha_core_interleaving}). In particular, the point $x$ is in $B^\beta_{(2\beta+1)r,k}(a)$ and thus also in $W^\beta_{(2\beta+1)r,k}(a)$. We get that $\Cr^\beta_{r,k}(A)\subseteq\aCr^\beta_{(2\beta+1)r,k}(A)$.
\end{proof}

\begin{figure}[h]
    \centering
    \includegraphics[width=0.3\textwidth]{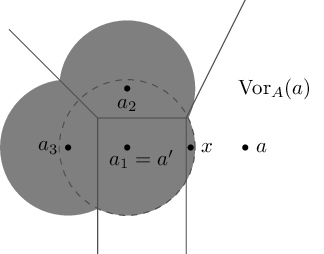}
    \caption{A diagram illustrating the proof of the second inclusion in \cref{alphatheorem} when $A\subseteq\mathbb{R}^2$, $\beta=1$ and $k=3$. Observe that the distance $d(a,a_3)$ can get arbitrarily close to $3r$ meaning that the multiplicative interleaving factor of $2\beta+1$ is the best we can hope for.}
    \label{fig:alpha_core_interleaving}
\end{figure}

Combining theorems \ref{corecovtheorem} and \ref{alphatheorem} gives us an interleaving between the multicover and Delaunay core bifiltrations. By considering the alpha-version of the $\Gamma_k$-balls ($\Gamma_k$ defined in \eqref{def:gamma}) we can make this interleaving even stricter. 
\begin{lemma}\label{lem:alphacoverlemma}
    Let $A\subseteq\mathbb{R}^n$ be a finite Euclidean subspace. The multicover bifiltration $\operatorname{Cov}(A)$ can be written as a union of $\Gamma_k$-Voronoi balls, i.e.,
    \begin{equation*}
        \operatorname{Cov}_{r,k}(A)=\bigcup_{a\in A}\left(B_{\Gamma_k}(a,r)\cap \operatorname{Vor}_A(a)\right).
    \end{equation*} 
\end{lemma}
\begin{proof}
    The $\Gamma_k$-Voronoi balls are contained in the $\Gamma_k$-balls, so their union is contained in $\operatorname{Cov}_{r,k}(A)$. Conversely, if $x\in \operatorname{Cov}_{r,k}(A)$, then, by \cref{covlemma}, $\Core^A_k(x) \leq r$ and $d(a',x) \leq r$ for some $a'\in A$. Now, if $x\in\operatorname{Vor}_A(a)$, then $d(a,x)\leq d(a',x) \leq r$ and so $x\in B_{\Gamma_k}(a,r)\cap \operatorname{Vor}_A(a)$.
\end{proof}

Despite our Delaunay construction not being homotopy equivalent to the core bifiltration, but only interleaved, the interleaving to the multicover bifiltration is preserved.

\begin{theorem}\label{alphacoremulticov}
        Let $A\subseteq\mathbb{R}^n$ be a finite Euclidean subspace and $\beta>0$. The Delaunay core bifiltration $\aCr^\beta(A)$ and the multicover bifiltration $\Cov(A)$ are interleaved as follows:
        \begin{equation*}
            \textit{i)}\, \aCr^\beta_{r,k}(A) \subseteq\Cov_{(1+\beta^{-1})r,k}(A), \quad \textit{ii)}\,\Cov_{r,k}(A) \subseteq\aCr^\beta_{\max\{1,2\beta\}r,k}(A).
        \end{equation*}
\end{theorem}
\begin{proof}
    $i)$ Using Theorem \ref{corecovtheorem}, we get the inclusion $\aCr^\beta_{r,k}(A)\subseteq\Cr^\beta_{r,k}(A)\subseteq\Cov_{(1+\beta^{-1})r,k}(A)$. 

    $ii)$ Let $x\in\Cov_{r,k}(A)$. By \cref{lem:alphacoverlemma}, this means that $\Core^A_k(x)\leq r$, so we can pick $a_1, \dots, a_{\lceil k\rceil}$ in $A$ with
    $d(a_i,x) \leq r$ for all $i=1,\dots, \lceil k\rceil$, and that there exists an $a\in A$ such that $x\in \Vor_A(a)$ and $d(a,x) \leq r$. In particular, we have $d(a,a_i)\leq d(a,x)+d(a_i,x) \leq 2r$, and so $\beta\Core^A_k(a) \leq 2\beta r$. Since $d(a,x) \leq r$ we get $x\in B^\beta_{\max\{1,2\beta\}r,k}(a)$, and since $x\in \Vor_A(a)$ we get, that $x$ is in $\aCr^\beta_{\max\{1,2\beta\}r,k}(A)$.
\end{proof}

In this section, we have established interleaving results for the Delaunay core bifiltration, enabling us to further derive Delaunay core stability as a corollary of multicover stability in the subsequent section.

\section{Stability for Core Bifiltration}\label{sec:stability}
\cite{blumberg2022stability} give a stability result for the multicover bifiltration with respect to the \emph{Prohorov distance}. In this section, we establish similar results for the core- and Delaunay core bifiltrations. Recall that, for a finite metric subset $S\subseteq (M,d)$ and $\delta\geq 0$, the $\delta$-thickening $S^\delta$ of $S$ is the set
\begin{equation*}
    S^\delta = \{x\in M\,|\,\textrm{there exists } s\in S \textrm{ such that } d(s,x) \leq \delta \}=\bigcup_{s\in S}B_d(s,\delta).
\end{equation*}
Now, for finite metric subspaces $A,B\subseteq M$, the \textbf{(counting) Prohorov distance} between them is
\begin{equation*}
    d_P(A,B) = \sup_{S\subseteq M\,\text{closed}} \inf \left\{\delta\geq 0\,\middle|\, {|S\cap A|} \leq {|S^\delta\cap B|} + \delta \textrm{ and } {|S\cap B|} \leq {|S^\delta\cap A|} + \delta\right\}.
\end{equation*}

By the triangle inequality, the $\delta$-thickening of a metric ball of radius $r$ is included in the metric ball of radius $r+\delta$, i.e.,\ $B_d(x,r)^\delta\subseteq B_d(x,r+\delta)$ for all $x\in M$.

\begin{theorem}[Multicover Stability {\cite[Rmk 3.2 for Thm 1.6]{blumberg2022stability}}]\label{multicoverstability}
    Let $A,B\subseteq (M,d)$ be finite metric subspaces and let $\delta > d_P(A,B)$. For all $r > 0$ and all $k>\delta$ we have inclusions
    \begin{equation*}
        \Cov_{r,k}(A)\subseteq\Cov_{r+\delta,k-\delta}(B)\quad\textrm{and}\quad \Cov_{r,k}(B)\subseteq\Cov_{r+\delta,k-\delta}(A).
    \end{equation*}
\end{theorem}
\begin{proof}
    We show the first inclusion, the second is symmetric. Let $x\in \Cov_{r,k}(A)$. This is true if and only if $k\leq |B_d(x,r)\cap A|$. Using our assumption and the fact that $B_d(x,r)^\delta \subseteq B_d(x,r+\delta)$ we have
    \begin{equation*}
        k\leq |B_d(x,r)\cap A|\leq |B_d(x,r)^\delta\cap B|+\delta\leq|B_d(x,r+\delta)\cap B|+\delta.
    \end{equation*}
    In particular, we get that $k-\delta\leq |B_d(x,r+\delta)\cap B|$, so $x\in\Cov_{r+\delta,k-\delta}(B)$.
\end{proof}

Using a similar approach, we get stability for the core bifiltration:

\begin{theorem}[Core Stability]
    Let $A,B\subseteq (M,d)$ be finite metric subspaces, fix $\beta>0$ and let $\delta > d_P(A,B)$. For all $r > 0$ and all $k>\delta$ we have inclusions
    \begin{equation*}
        \Cr^\beta_{r,k}(A)\subseteq\Cr^\beta_{r',k-\delta}(B)\quad\textrm{and}\quad \Cr^\beta_{r,k}(B)\subseteq\Cr^\beta_{r',k-\delta}(A),
    \end{equation*}
    where $r'=\max\{2(r+\beta\delta),(1+\beta^{-1})r+\delta\}$.
\end{theorem}
\begin{proof}
    Let $x\in \Cr^\beta_{r,k}(A)$. This is true if and only if there exists an $a\in A$ such that $d(a,x) \leq r$ and $k\leq |B_d(a,r/\beta)\cap A|$. If $a\in A$ is such a point, then 
    \begin{equation}\label{eq:stabilityargument}
        k\leq |B_d(a,r/\beta)\cap A|\leq |B_d(a,r/\beta)^\delta\cap B|+\delta\leq|B_d(a,r/\beta+\delta)\cap B|+\delta.
    \end{equation}
    Therefore, we have $0< k-\delta \leq |B_d(a,r/\beta+\delta)\cap B|$, so let $b$ be an element in the intersection $B_d(a,r/\beta+\delta)\cap B$. By the triangle inequality, we have $B_d(a,r/\beta+\delta)\subseteq B_d(b,2r/\beta+2\delta)$ and consequently, $k-\delta\leq |B_d(b,2r/\beta+2\delta)\cap B|$ which is to say that $\beta\Core^B_{k-\delta}(b)\leq2(r+\beta\delta)$. Furthermore, we have that $d(b,x)\leq d(b,a)+d(a,x) \leq r/\beta+\delta+r=(1+\beta^{-1})r+\delta$. The other direction follows from the symmetry of the Prohorov distance.
\end{proof}

Combining theorems \ref{multicoverstability} and \ref{alphacoremulticov}, we get the following stability result for the Delaunay core bifiltration.
\begin{corollary}[Delaunay Core Stability]\label{cor:alphacorestability}
    Let $A,B\subseteq\mathbb{R}^n$ be finite Euclidean subspaces, fix $\beta>0$ and let $\delta > d_P(A,B)$. For all $r > 0$ and $k>\delta$, we have inclusions 
    \begin{equation*}
        \aCr^\beta_{r,k}(A)\subseteq\aCr^\beta_{r',k-\delta}(B)\quad\textrm{and}\quad \aCr^\beta_{r,k}(B)\subseteq\aCr^\beta_{r',k-\delta}(A).
    \end{equation*} 
    where $r'=\max\{1,2\beta\}((1+\beta^{-1})r+\delta)$.
    \qed
\end{corollary}


\cite[Def.\ 2.13]{blumberg2022stability} also consider a \emph{normalized} version of the multicover bifiltration, which we will denote $\overline{\Cov}(A)$ given in filtration degree $(r,s)$ by $\overline{\Cov}_{r,s}(A)=\Cov_{r,s|A|}(A)$. In practice, we are often in a situation where $|A|\neq|B|$. It then makes more sense to consider the normalized version, making the second parameter comparable. Similarly, we can define normalized versions of the core- and Delaunay core bifiltrations by letting
\begin{equation*}
    \overline{\Cr}^\beta_{r,s}(A)=\Cr^\beta_{r,s|A|}(A)\quad\text{and}\quad
    \overline{\aCr}^\beta_{r,s}(A)=\aCr^\beta_{r,s|A|}(A).
\end{equation*}

The \textbf{normalized Prohorov distance} between finite metric subspaces $A,B\subseteq M$ is defined as
\begin{equation*}
    \resizebox{.95\hsize}{!}{$d^N_P(A,B) = \underset{S\subseteq M\,\text{closed}}{\sup} \inf \left\{\delta\geq 0\,\middle|\, \frac{|S\cap A|}{|A|} \leq \frac{|S^\delta\cap B|}{|B|} + \delta \textrm{ and } \frac{|S\cap B|}{|B|} \leq \frac{|S^\delta\cap A|}{|A|} + \delta\right\}$}.
\end{equation*}
As in the unnormalized case, the stability for normalized multicover bifiltration in \cite[Thm.\ 1.6]{blumberg2022stability} also extends to a stability for the normalized core- and Delaunay core bifiltrations.

\begin{theorem}[Normalized Stability]\label{cor:normalizedstability}
    Let $A,B\subseteq\mathbb{R}^n$ be finite Euclidean subspaces, $\beta>0$ and $\delta > d^N_P(A,B)$. For all $r > 0$ and $s>\delta$, we have inclusions 
    \begin{align*}
        \overline{\Cov}_{r,s}(A)&\subseteq\overline{\Cov}_{r+\delta,s-\delta}(B), &\quad\overline{\Cov}_{r,s}(B)&\subseteq\overline{\Cov}_{r+\delta,s-\delta}(A); \\
        \overline{\Cr}^\beta_{r,s}(A)&\subseteq\overline{\Cr}^\beta_{r',s-\delta}(B), &\quad\overline{\Cr}^\beta_{r,s}(B)&\subseteq \overline{\Cr}^\beta_{r',s-\delta}(A); \\
        \overline{\aCr}^\beta_{r,s}(A)&\subseteq \overline{\aCr}^\beta_{r'',s-\delta}(B), &\quad \overline{\aCr}^\beta_{r,s}(B)&\subseteq \overline{\aCr}^\beta_{r'',s-\delta}(A),
    \end{align*}
    where $r'=\max\{2(r+\beta\delta),(1+\beta^{-1})r+\delta\}$ and $r''=\max\{1,2\beta\}((1+\beta^{-1})r+\delta)$.
\end{theorem}

\section{Experiments} \label{sec:experiments}

This section is divided into two main parts. In the first, we compute persistent homology along different lines (or slices) following common practice in topological data analysis \cite{rolle2024stable}, and present results on five noisy point cloud samples. In the second, we compute persistent homology for the full Delaunay core bifiltration. Additionally, we compare our results with a sublevel Delaunay-\v{C}ech bifiltration \cite{alonso2024delaunay} and the multicover bifiltration computed via the rhomboid tiling bifiltration \cite{Edelsbrunner2021,Edelsbrunner2023,Corbet2023}.

\subsection{Computing Persistence Along a Line} \label{sec:persistence_on_line}

We have implemented code to compute the persistent homology of the core and Delaunay core bifiltrations along a line in parameter space. The implementation, available at \url{https://github.com/odinhg/core}, includes code demonstrating its use, and code to reproduce the experiments reported in this section. The implementation uses the GUDHI library \cite{gudhi:urm, gudhi:FilteredComplexes,gudhi:AlphaComplex}, which allows for easy computation of the corresponding persistent homology modules and the bottleneck distances between them.

We compute the persistent homology of the core and Delaunay core bifiltration along the line $k=g(r)=-\frac{k_{\text{max}}}{r_{\text{max}}}r+k_{\text{max}}$ where $k_{\text{max}}$ and $r_{\text{max}}$ are positive real numbers. Note that for non-integer values of $k$, we round up to the nearest integer, i.e., we compute persistence along the piece-wise constant function $k=\lceil g(r)\rceil$. We also compute the core and Delaunay core bifiltrations for $k$ fixed ($r_\text{max}=+\infty$). 

In our experiments, we choose $r_{\text{max}}$ to be the diameter of the input point cloud and set the parameter $k_\text{max}=\max\left\{1, \lfloor s_\text{max} \vert X\vert\rfloor\,\right\}$ with $s_\text{max}\in\{0, 0.001, 0.01, 0.1\}$ so that $k_\text{max}$ scales with the number of points in the input point cloud. Similarly, in the fixed-$k$ case, we let $k=\max\left\{1, \lfloor s \vert X\vert\rfloor\right\}$.

We consider five point cloud datasets in our experiments: \textbf{Torus 1}, \textbf{Torus 2}, \textbf{Sphere}, \textbf{Circle} and \textbf{Circles}. See \cref{fig:point_clouds} for a visualization of four of these datasets. We construct our point cloud datasets as follows: For an underlying manifold $M\subseteq\mathbb{R}^d$ (see \cref{table:dataset_manifolds} for an overview of the manifolds used for the different datasets), we first uniformly sample a point cloud of size $n$ from the manifold $M$. We then perturb the points according to a normal distribution having mean $\mu=0$ and standard deviation $\sigma=0.07$ to obtain a perturbed sample $Z$. In the last step, we uniformly sample $m$ points $Y$ from the smallest axis parallel hyperbox containing $Z$, and obtain our point cloud dataset $X=X(M, m, n, \sigma)=Z\cup Y$. We think of $Z$ as a noisy signal and $Y$ as background noise. \Cref{fig:torus_persistence_diagrams} shows persistence diagrams for the Delaunay core bifiltration for various values of $k_\text{max}$ for the \textbf{Torus 2} dataset.

\begin{table}[h]
\centering
\begin{tabular}{rl}
\toprule
\textbf{Dataset name} & \textbf{Manifold $M$} \\ \midrule
\textbf{Torus 1} & The 2-torus embedded in $\mathbb{R}^3$.           \\
\textbf{Torus 2} & The Clifford torus $S^1\times S^1$ in $\mathbb{R}^4$.           \\
\textbf{Sphere}  & The 2-sphere $S^2\subseteq\mathbb{R}^3$.           \\
\textbf{Circle}  & The 1-sphere $S^1\subseteq\mathbb{R}^2$.           \\
\textbf{Circles} & The union of two circles with radii $0.5$ and $1$ in $\mathbb{R}^2$.\\    
\bottomrule
\end{tabular}
\caption{The underlying manifolds used to generate the five datasets.}
\label{table:dataset_manifolds}
\end{table}

As a ground truth for a sample $X=X(M,m,n,\sigma)$, we use the \v{C}ech persistent homology of a uniform sample $X_t=X(M,0,m+n,0)$ from $M$ of size $n+m$ without any perturbation or added noise. In practice, we compute the \v{C}ech persistence using Delaunay core with $k=1$. To compare the Delaunay core persistence of the noisy sample against the ground truth \v{C}ech persistence, we compute the bottleneck distance between their corresponding persistence diagrams. The computed bottleneck distances for both $k=g(r)$ and $k$ fixed are listed in the tables \ref{table:bottleneck_distances_n_10000_m_10000}, \ref{table:bottleneck_distances_n_10000_m_1000} and \ref{table:bottleneck_distances_n_10000_m_100} for the different combinations of $n$ and $m$. In addition, we report runtimes for computing Delaunay core persistence on the \textbf{Torus 1} dataset in \cref{table:runtimes_torus_1} for different choices of $k$ and point cloud sizes. In the aforementioned experiments, we use $\beta=1$. We also perform experiments to examine how the value of $\beta$ affects the bottleneck distance to the ground truth diagram. The plots in \cref{fig:bottleneck_beta_torus_2} and \cref{fig:beta_bottleneck_distance_plots} show the change in bottleneck distance when varying the parameter $\beta\in\{0.125, 0.25, 0.5, 1, 2, 4, 8\}$ while keeping the parameters $n=10000$, $m=1000$, $\sigma=0.07$ and $s_{\text{max}}=0.01$ fixed.

Based on our experiments, we see that choosing $s>0$, corresponding to $k>1$, in most cases gives persistence diagrams closer to the ground truth in terms of the bottleneck distance than for $k=1$ (\v{C}ech persistence). Moreover, in most cases, the persistence along a sloped line yields slightly smaller bottleneck distances than persistence along lines with constant $s$. However, the difference between the bottleneck distances in these two cases is small in all of our experiments. We also note that the optimal choice of $s$ depends both on the homological dimension and on the shape of the point clouds. Based on the plots in \cref{fig:bottleneck_beta_torus_2} and \cref{fig:beta_bottleneck_distance_plots}, the choice of $\beta$ appears to influence the bottleneck distance to the ground truth diagram. But also for $\beta$, the optimal choice seems to depend on the homological dimension and the point cloud shape.

\begin{figure}[h]
\centering

\includegraphics[width=0.8\textwidth]{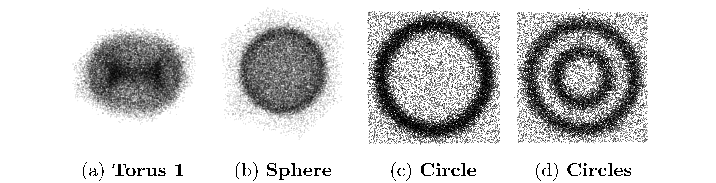}

\caption{Four of the point cloud datasets used in our experiments with $\sigma=0.07$ (standard devitation), $n=20000$ (signal) and $m=10000$ (noise).}
\label{fig:point_clouds}
\end{figure}

\begin{figure}[h]
\centering

\includegraphics[width=\textwidth]{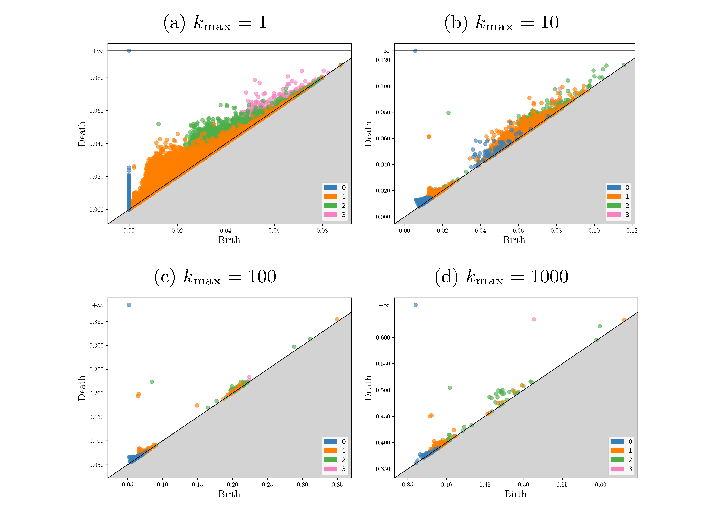}

\caption{Persistence diagrams for the Delaunay core bifiltration ($\beta=1$) of the \textbf{Torus 2} (Clifford torus) dataset with $r_\text{max}=\operatorname{diam}(X)\approx 4.10$. In this example, we have a $1:1$ signal-to-noise ratio with $n=m=10000$, and observe that increasing $k_\text{max}$ strengthens the separation between noise and those persistence pairs we expect to see for a torus. The different colours correspond to the different homological dimensions.}
\label{fig:torus_persistence_diagrams}
\end{figure}

\begin{figure}[h]
\centering
\includegraphics[width=0.4\textwidth]{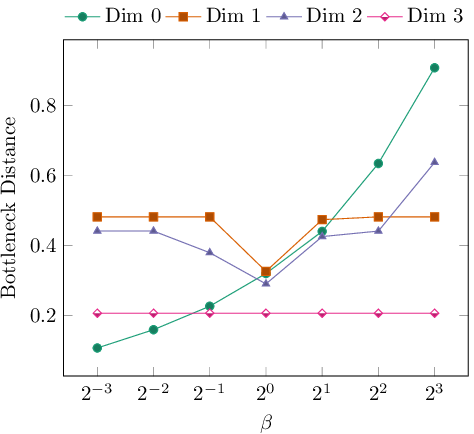}
\caption{Bottleneck distance to the ground truth diagram for the \textbf{Torus 2} dataset when we vary the value of $\beta$. Similar plots for the other datasets can be found in \cref{fig:beta_bottleneck_distance_plots}.}
\label{fig:bottleneck_beta_torus_2}
\end{figure}

\clearpage

\subsection{Computing the Full Delaunay Core Bifiltration} \label{sec:computing_the_full_delaunay_core_bifiltration}

While computing persistence along a line is an efficient way of extracting topological descriptors from noisy point clouds, it is not obvious how one should choose the best line in practical applications. Therefore, one might want to compute the full bifiltration instead. Our implementation of the Delaunay core bifiltration is part of the \texttt{multipers} library \cite{multipers}, an extensive framework for computing multipersistent homology and related invariants. 

The \emph{Multiparameter Module Approximation (MMA)} algorithm introduced in \cite{loiseaux2022fast} is a method to compute approximations of multipersistence modules by gluing together one-dimensional slices. This provides a topological invariant for multiparameter filtrations. We note that the MMA algorithm depends on a choice of diagonal lines to compute one-dimensional persistence along (the fibered barcode), and a matching function for matching pairs of bars corresponding to the same underlying summand. The output of the MMA algorithm is an interval decomposable module, and can thus be visualized by using different colors for the summands. For more details on the MMA algorithm and its properties, see \cite{loiseaux2024multiparameter} and \cite{loiseaux2022fast}. 

Another invariant for persistence modules is the \emph{Hilbert function}. Given a $P$-filtered simplicial complex $K_\bullet$, the Hilbert function $\operatorname{HF}\colon P\to\mathbb{N}$ is defined as $p\mapsto\dim(H_\bullet(K_p))$. Fixing a homological dimension, the Hilbert function of a bifiltration can be visualized as a grayscale heat map. In our experiments, we compute and visualize module approximations with the MMA algorithm, and Hilbert functions, using the \texttt{multipers} library.

To compute the Delaunay core bifiltration for a point cloud $A \subseteq \mathbb{R}^d$, we begin by constructing the alpha complex. Then, for each simplex in the alpha complex, we assign a minimal filtration value for each $k$. For large point clouds, computational complexity can be reduced by specifying a list of $k$ values rather than considering every possible value of $k$. See \cref{alg:delaunay_core} in the appendix for the pseudocode. For all the following experiments, we set $\beta=1$ in the Delaunay core bifiltration.

\subsubsection{Delaunay Core Persistence of Uniform Noise}\label{sec:delaunay_core_persistence_of_noise}

We first analyze a dataset of uniform noise, looking at its MMA and Hilbert function plots. This helps us identify the characteristic shape of the uniform background noise in later experiments. We compute $H_0$ and $H_1$ persistent homology for a point cloud of $1000$ points uniformly sampled from $[-1,1]^2$. This reveals a thin curved artifact in the MMA and Hilbert function plots shown in \cref{fig:uniform_noise_delaunay_core} which also appears in the analysis of the subsequent datasets.

\begin{figure}[H]
\centering
\begin{minipage}{0.30\textwidth}
\centering
\includegraphics[width=\linewidth]{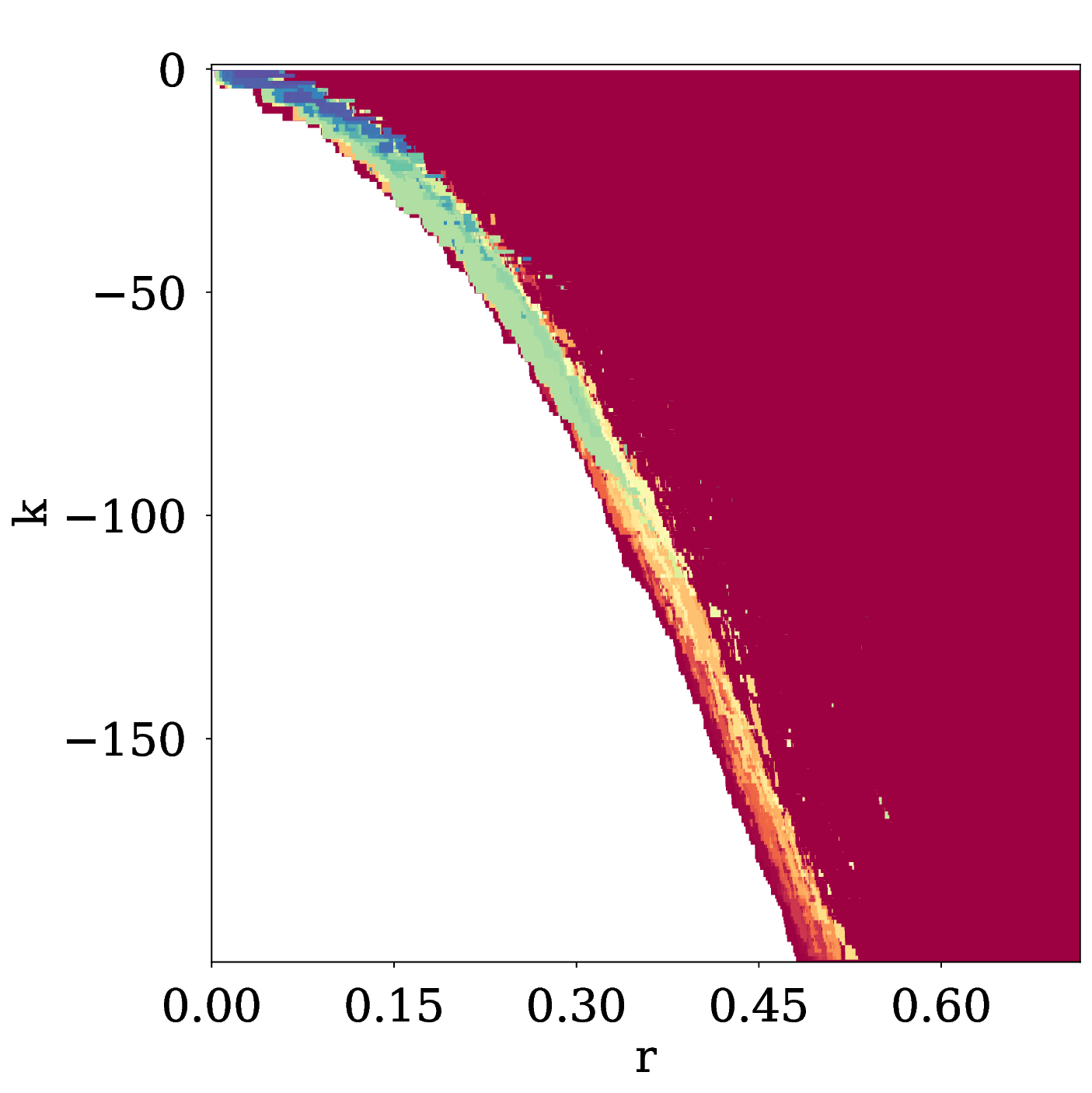}
\end{minipage}%
\begin{minipage}{0.30\textwidth}
\centering
\includegraphics[width=\linewidth]{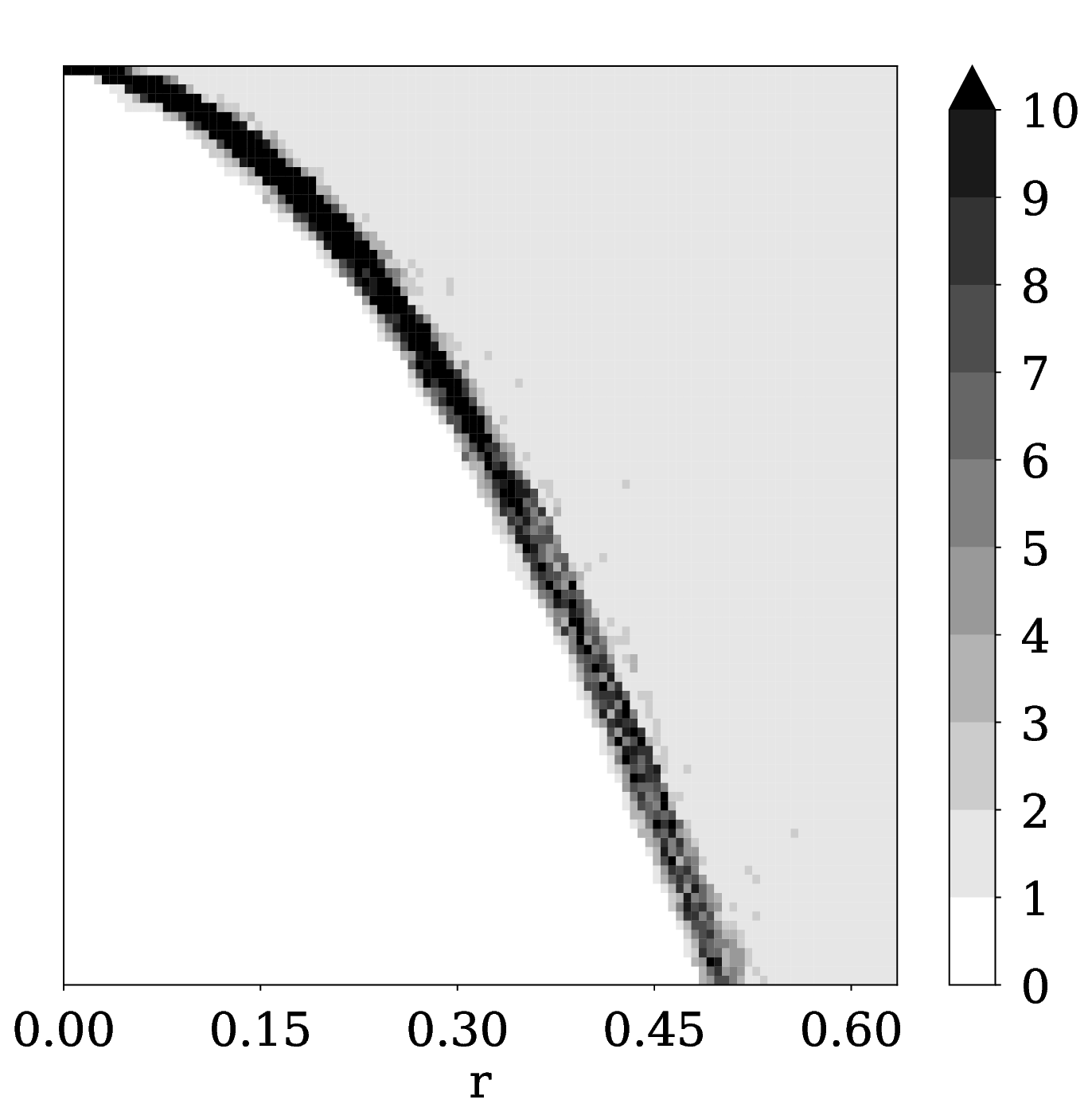}
\end{minipage}
\begin{minipage}{0.30\textwidth}
\centering
\includegraphics[width=\linewidth]{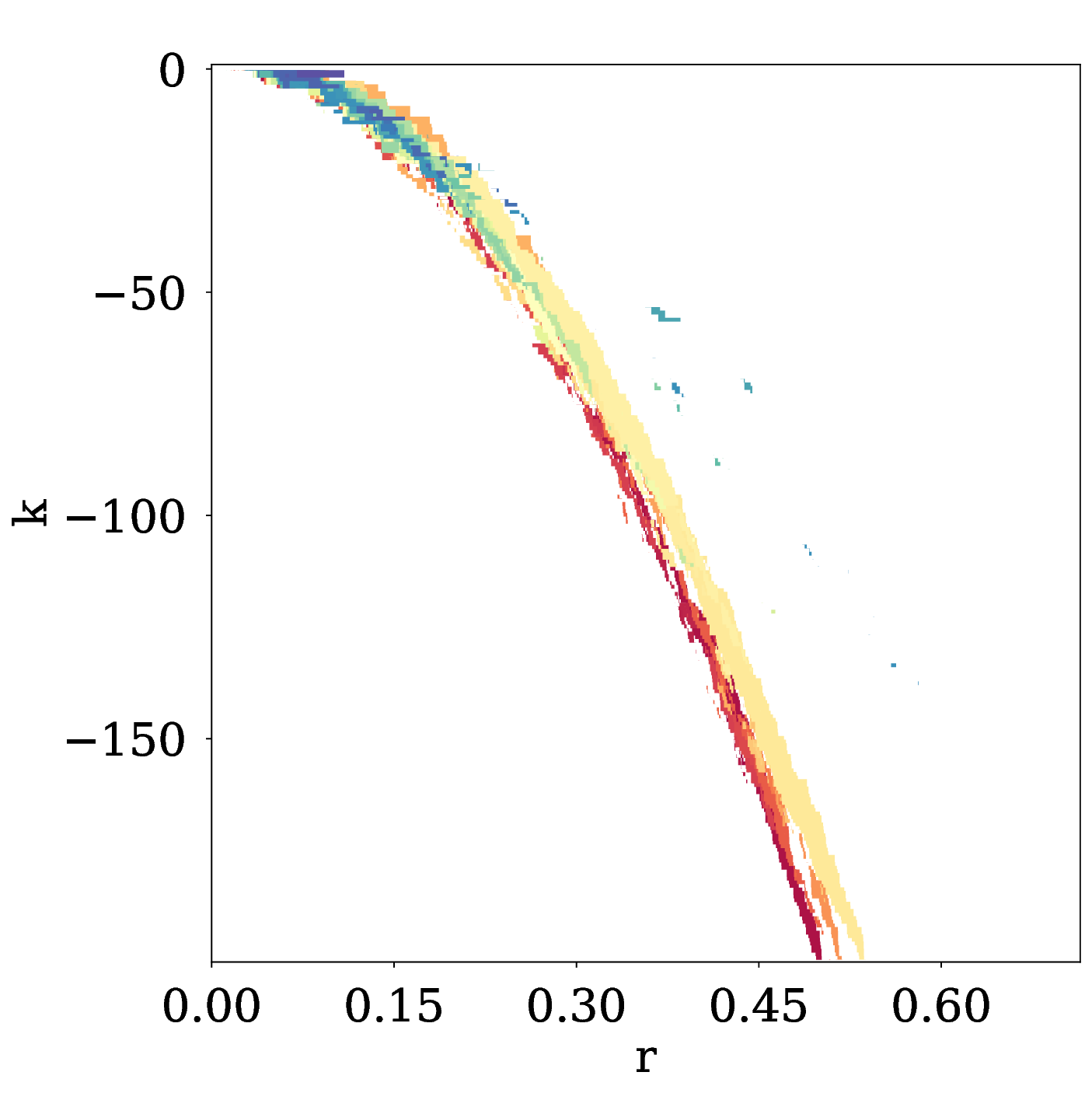}
\end{minipage}%
\begin{minipage}{0.30\textwidth}
\centering
\includegraphics[width=\linewidth]{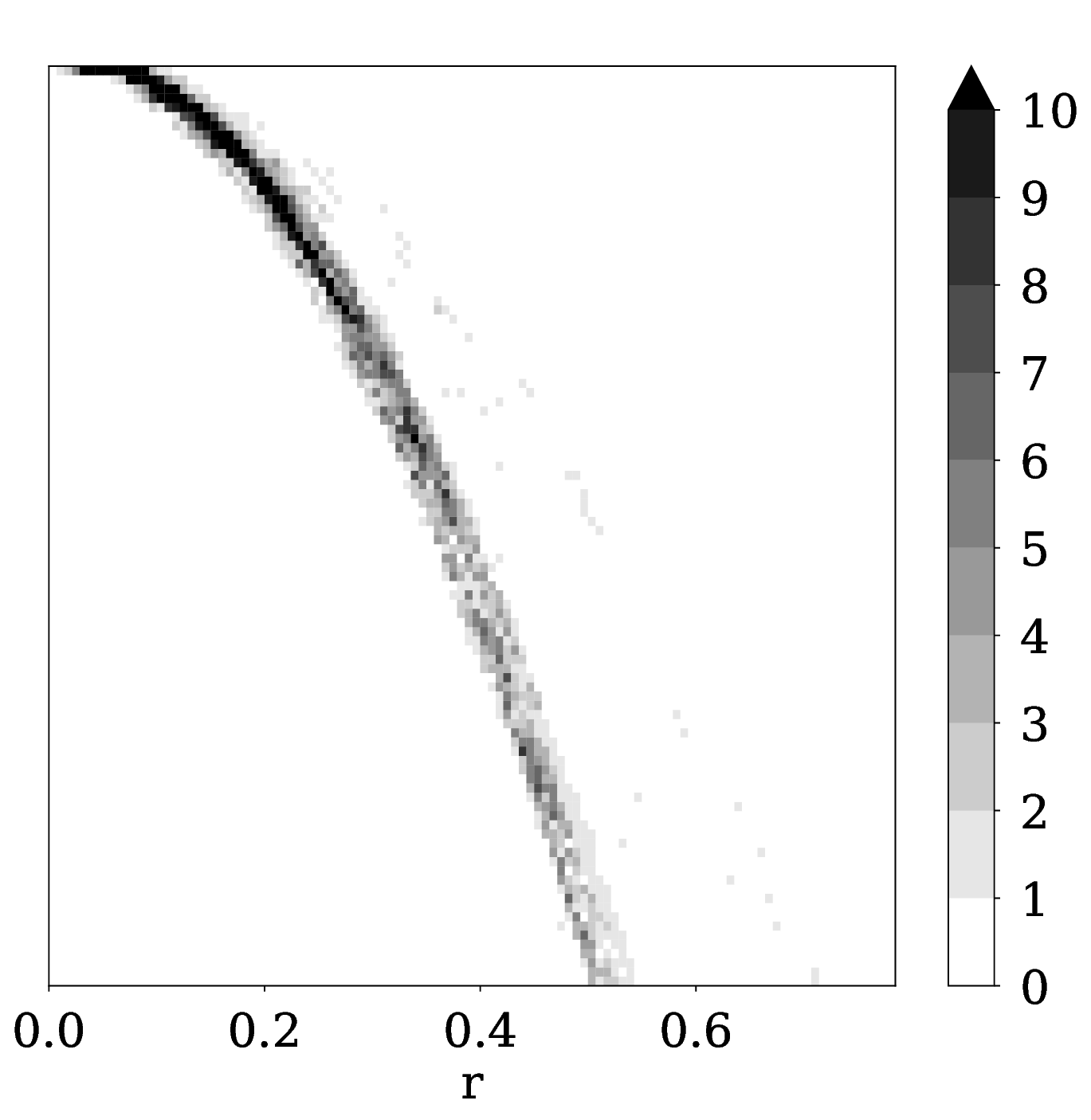}
\end{minipage}
\caption{Plots visualizing the MMA (left) and the Hilbert functions (right) for the Delaunay core bifiltration of a point cloud uniformly sampled from $[-1,1]^2$. The first and second rows correspond to homological dimensions $0$ and $1$, respectively.}
\label{fig:uniform_noise_delaunay_core}
\end{figure}

\subsubsection{Comparing with a Sublevel Delaunay-\v{C}ech Bifiltration}\label{sec:comparing_to_sublevel_delaunay_cech}

To demonstrate the performance of the Delaunay core bifiltration on a more complex dataset, we use the three annuli dataset, provided by the \texttt{multipers} library. This dataset consists of a sample from three annuli of different scales and densities, with added noise. For the Delaunay core bifiltration (see \cref{fig:delaunay_core_three_annulus_persistence}), we use $2000$ points equally divided between signal and noise, whereas for the sublevel Delaunay-\v{C}ech bifiltration (see \cref{fig:function_delaunay_three_annulus_persistence}) we use $20000$ points in total. The bottleneck in our implementation currently lies in the approximation of the multipersistence module. We anticipate improvements in this area in the future, which would accelerate the overall Delaunay core bifiltration computation.

In this section, we compute the Delaunay core bifiltration for $k\in\{1, 2, \dots, 200\}$. For the sublevel Delaunay-\v{C}ech bifiltration, we use a codensity function $\gamma \colon A \to \mathbb{R}$, defined as the logarithm of a Gaussian kernel density estimate with bandwidth parameter $0.10$. For easier visualization, summands with lifetimes below $0.01$ for Delaunay core persistence and $0.001$ for sublevel Delaunay-\v{C}ech persistence are excluded from the plots.

\begin{figure}[H]
\centering
\begin{minipage}{0.33\textwidth}
\centering
\includegraphics[width=\linewidth]{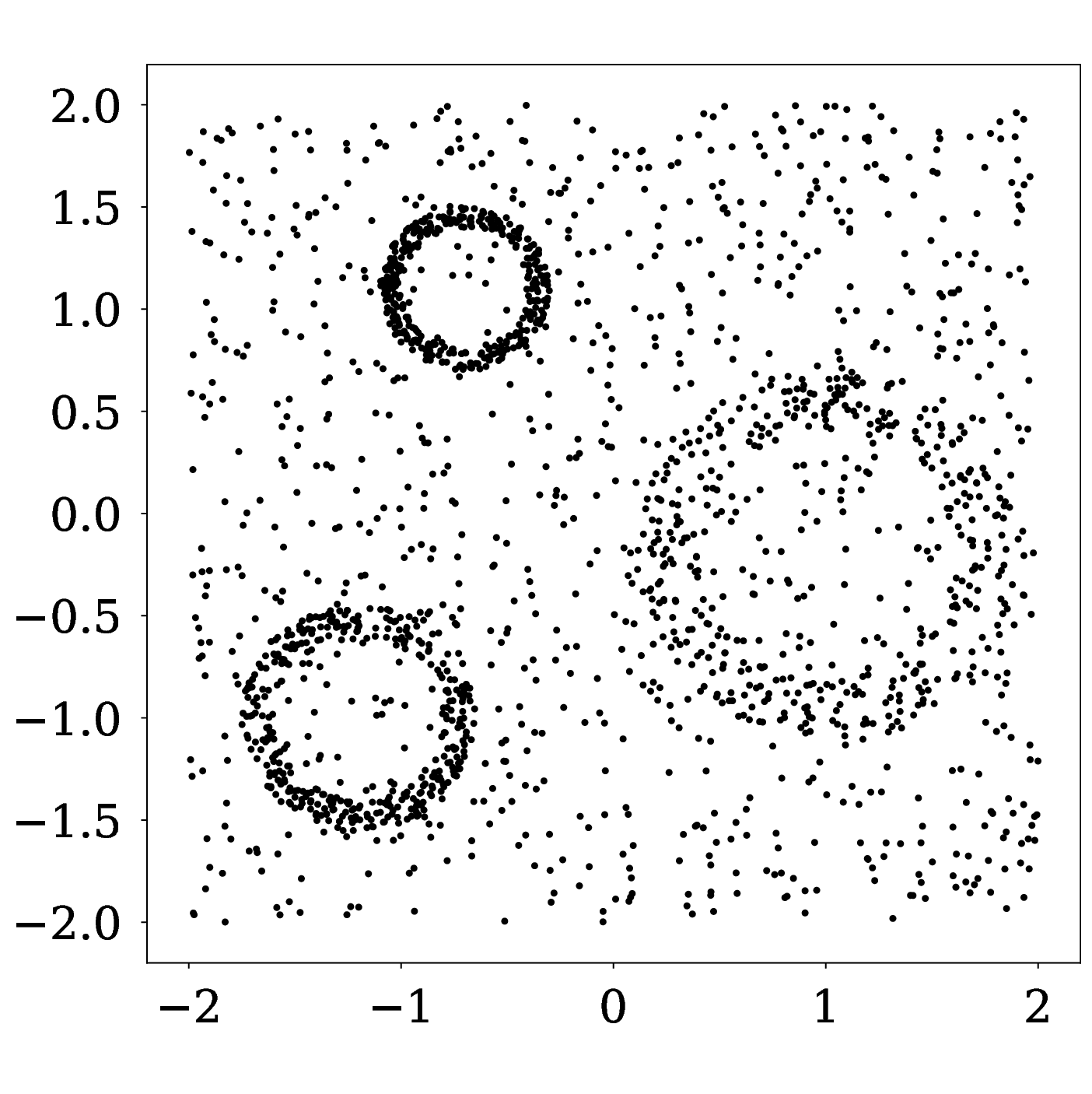}
\end{minipage}%
\begin{minipage}{0.33\textwidth}
\centering
\includegraphics[width=\linewidth]{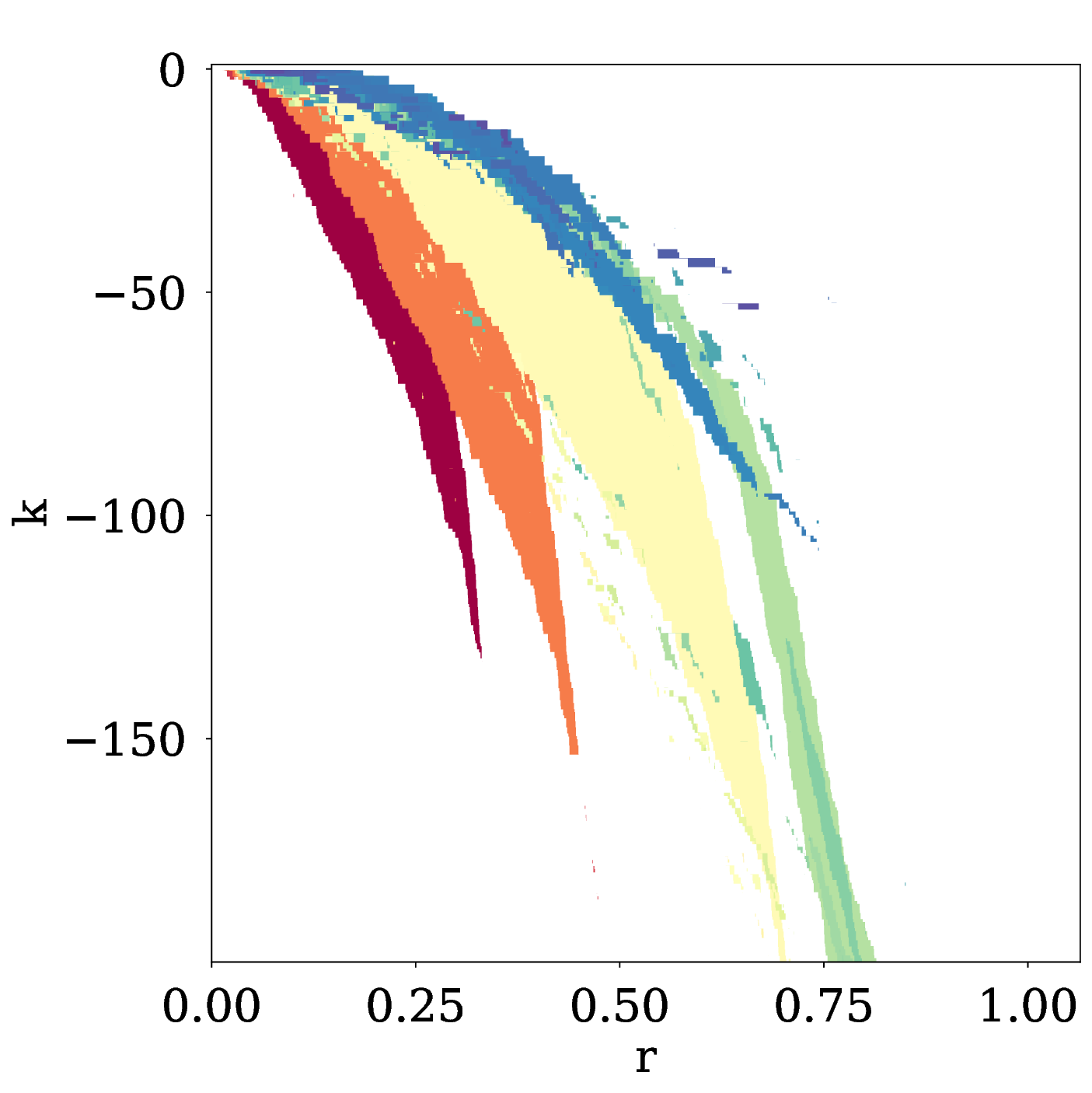}
\end{minipage}%
\begin{minipage}{0.33\textwidth}
\centering
\includegraphics[width=\linewidth]{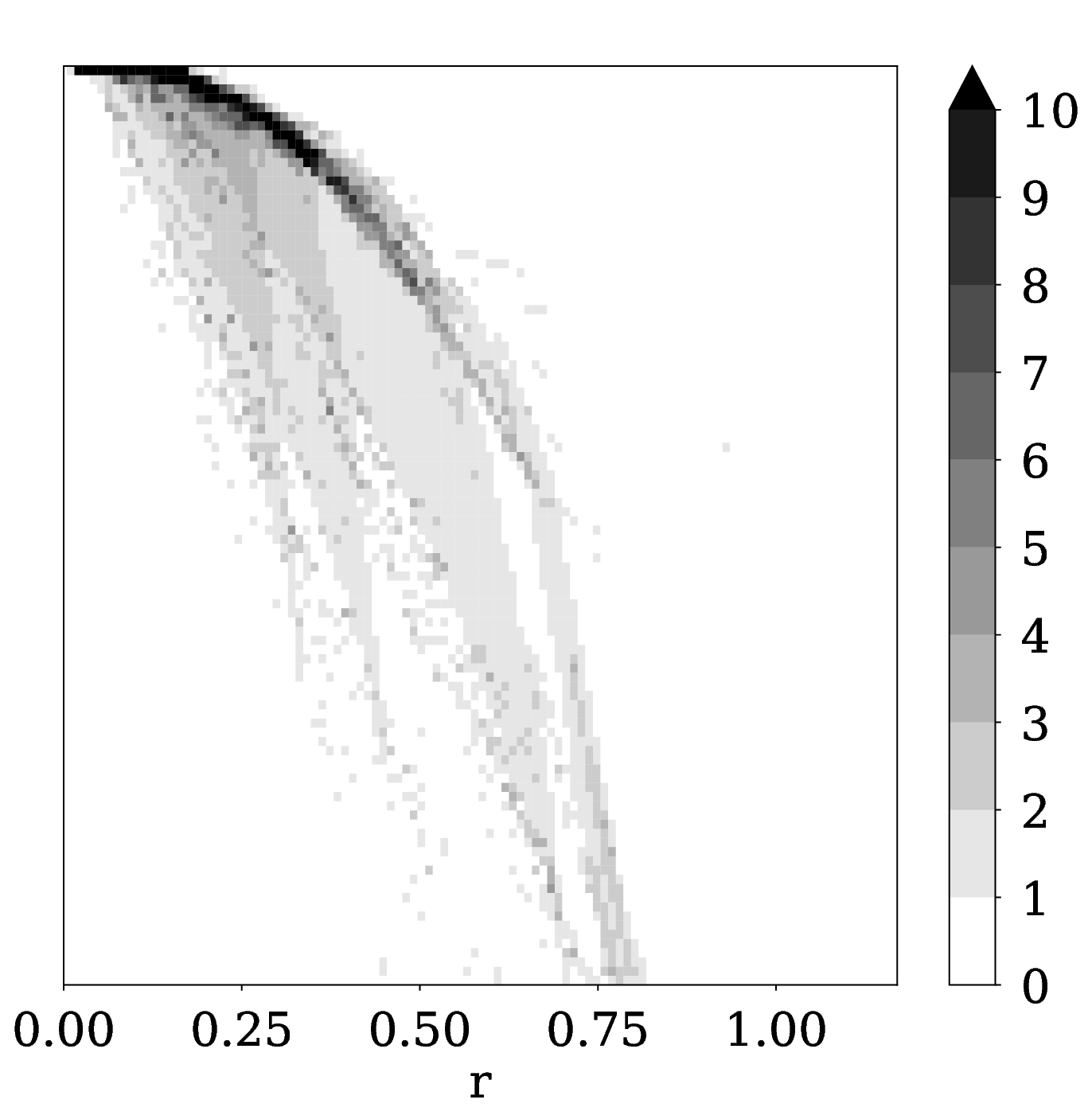}
\end{minipage} 
\caption{Point cloud sampled from three annuli with added noise (left), together with plots visualizing the MMA (middle) and the Hilbert function (right) for the Delaunay core bifiltration in homological dimension $1$.}
\label{fig:delaunay_core_three_annulus_persistence}
\end{figure}

\begin{figure}[H]
\centering
\begin{minipage}{0.33\textwidth}
\centering
\includegraphics[width=\linewidth]{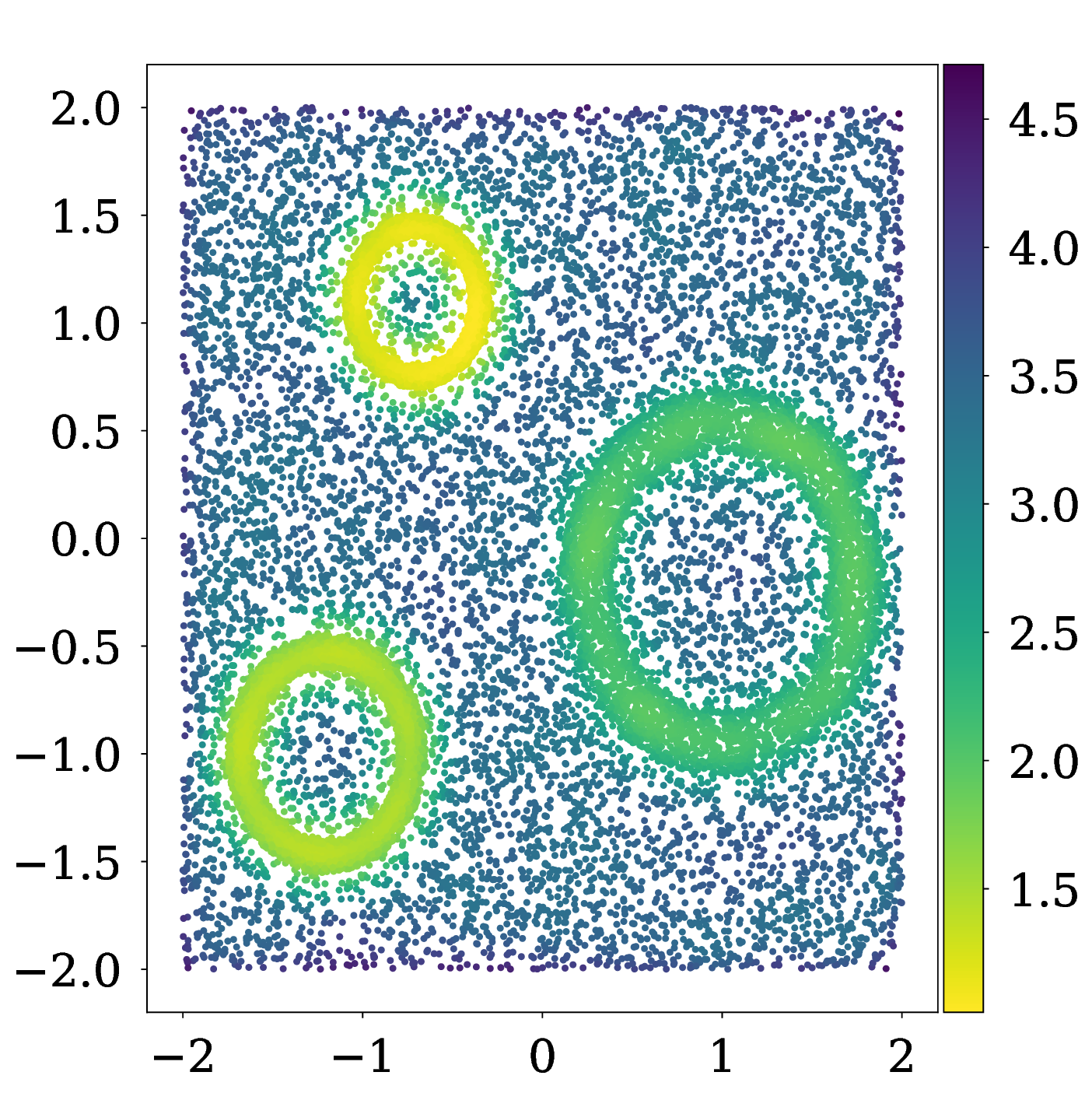}
\end{minipage}%
\begin{minipage}{0.33\textwidth}
\centering
\includegraphics[width=\linewidth]{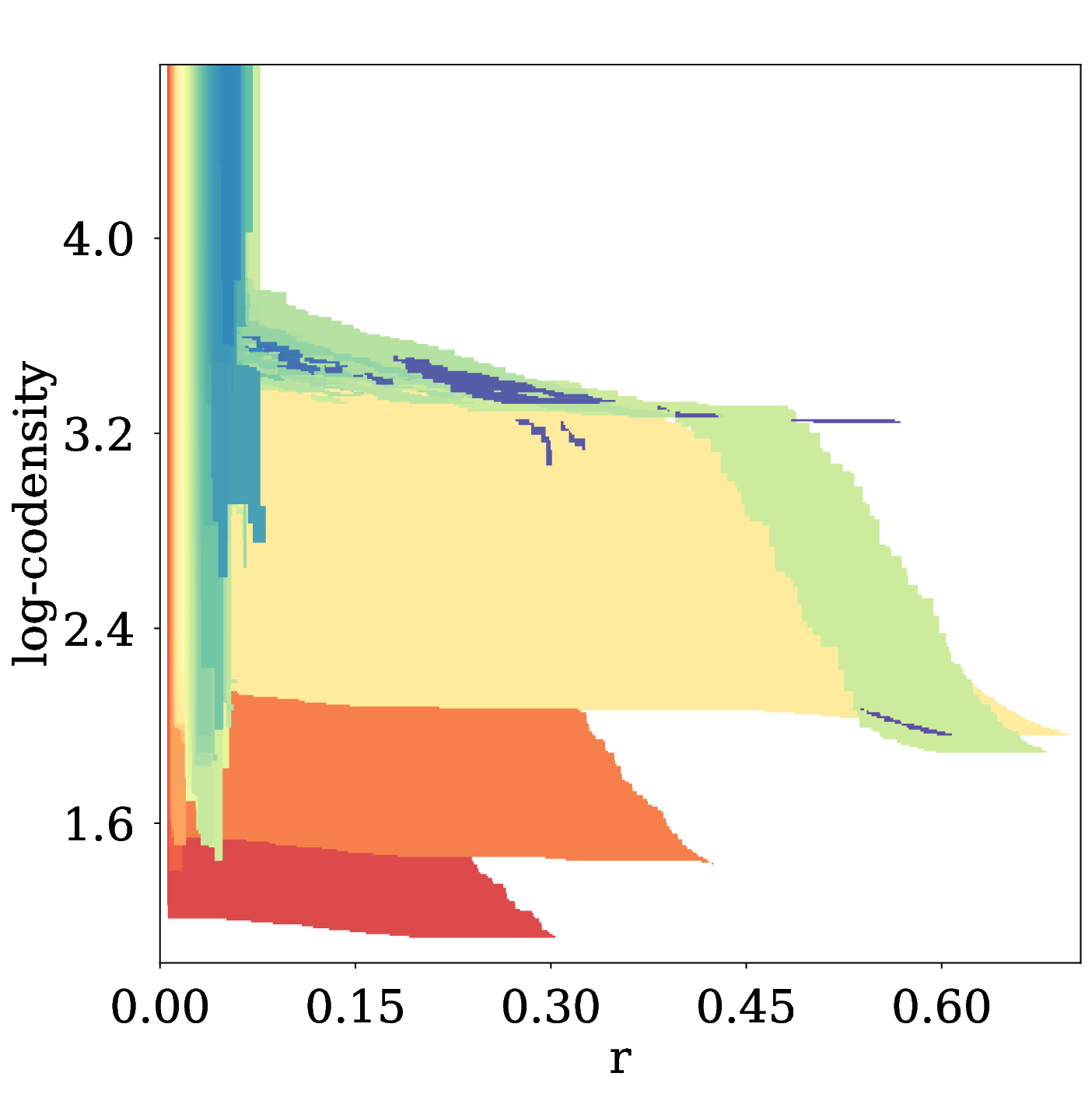}
\end{minipage}%
\begin{minipage}{0.33\textwidth}
\centering
\includegraphics[width=\linewidth]{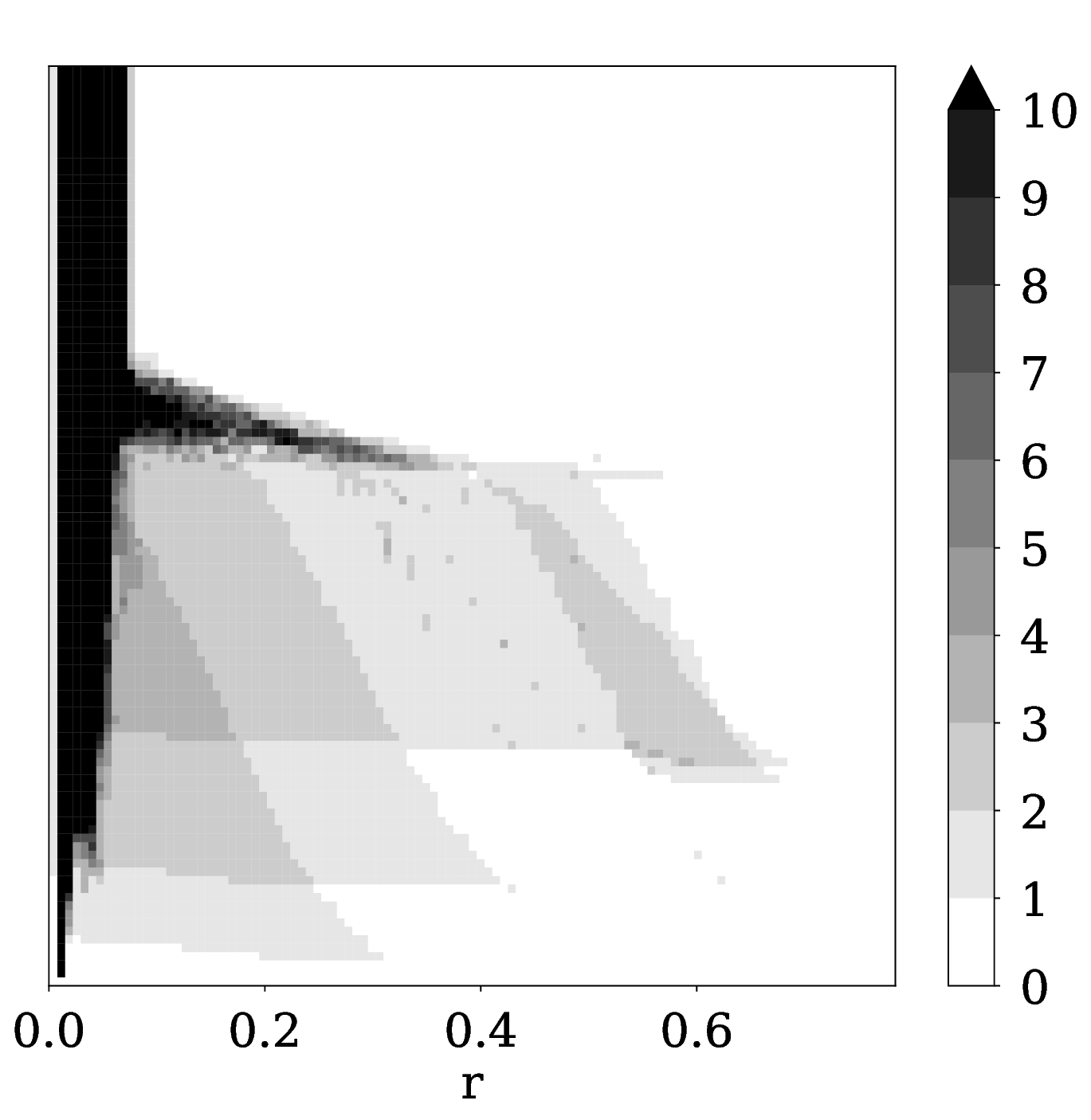}
\end{minipage} 
\caption{Point cloud sampled from three annuli with added noise (left), colored by the codensity function $\gamma$, together with plots visualizing the MMA (middle) and the Hilbert function (right) for the sublevel Delaunay-\v{C}ech bifiltration of $\gamma$ in homological dimension $1$.}
\label{fig:function_delaunay_three_annulus_persistence}
\end{figure}

Next, we analyze a synthetic clustering dataset introduced in \cite{mcinnes2017hdbscan} and also used in \cite{rolle2024stable}. It consists of $2309$ points in the plane, with six clusters of varying sizes, shapes, and densities, along with added noise. We include this dataset to examine how the Delaunay core bifiltration reflects the structure of connected components through their persistence across scale and density parameters. For this purpose, we compute and visualize $H_0$ persistence for both the Delaunay core (see \cref{fig:delaunay_core_clusterable_data_point_cloud}) and sublevel Delaunay-\v{C}ech bifiltrations (see \cref{fig:function_delaunay_clusterable_data_point_cloud}). Both bifiltrations reveal multiple connected components in persistent homology, suggesting the presence of clusters. Although the precise number of clusters is not obvious, the visualizations offer a good estimate. We also note that the results for the sublevel Delaunay-\v{C}ech bifiltration are sensitive to the choice of the bandwidth parameter, which may be suboptimal for this dataset.

\begin{figure}[H]
\centering
\begin{minipage}{0.33\textwidth}
\centering
\includegraphics[width=\linewidth]{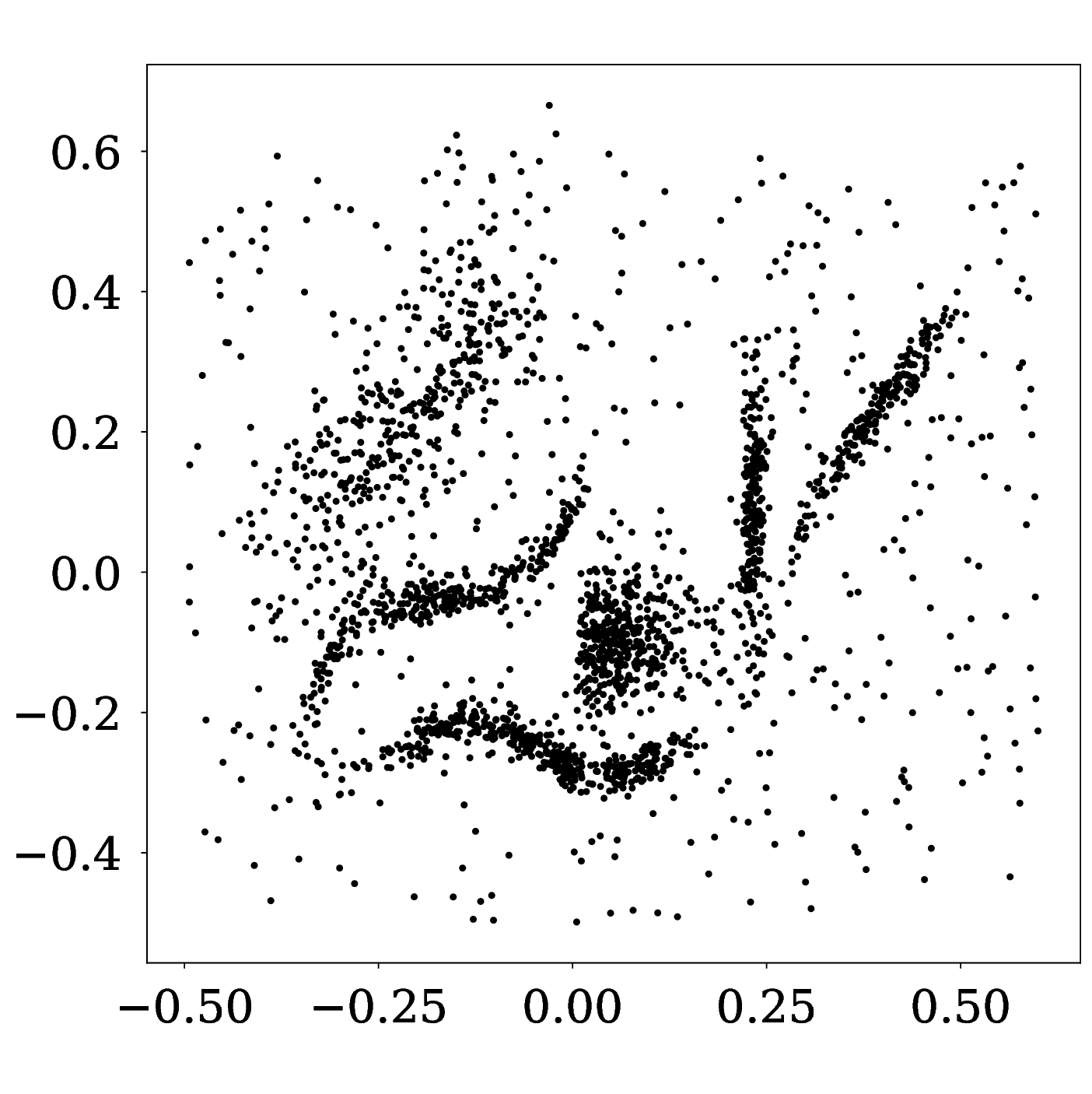}
\end{minipage}%
\begin{minipage}{0.33\textwidth}
\centering
\includegraphics[width=\linewidth]{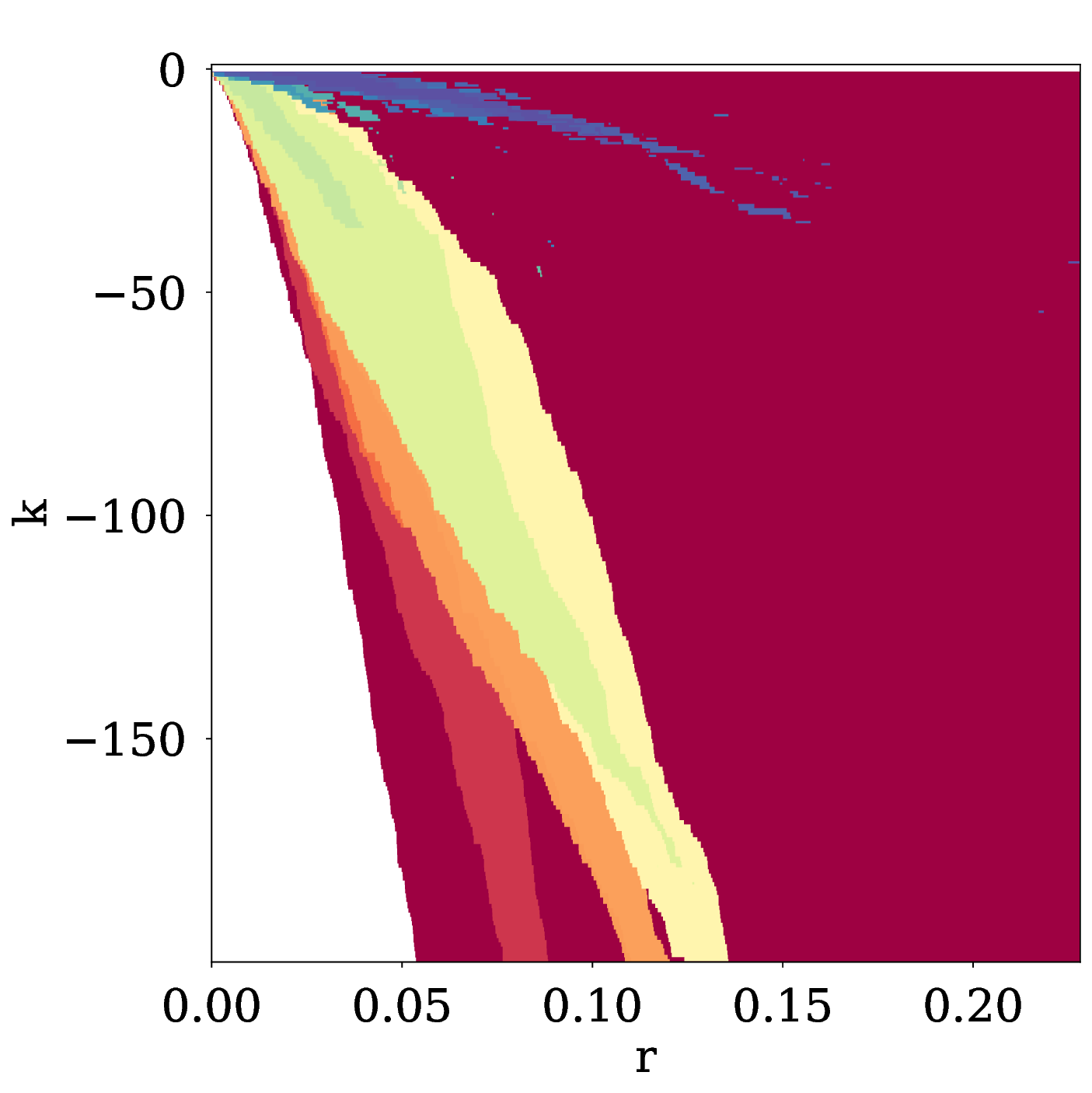}
\end{minipage}%
\begin{minipage}{0.33\textwidth}
\centering
\includegraphics[width=\linewidth]{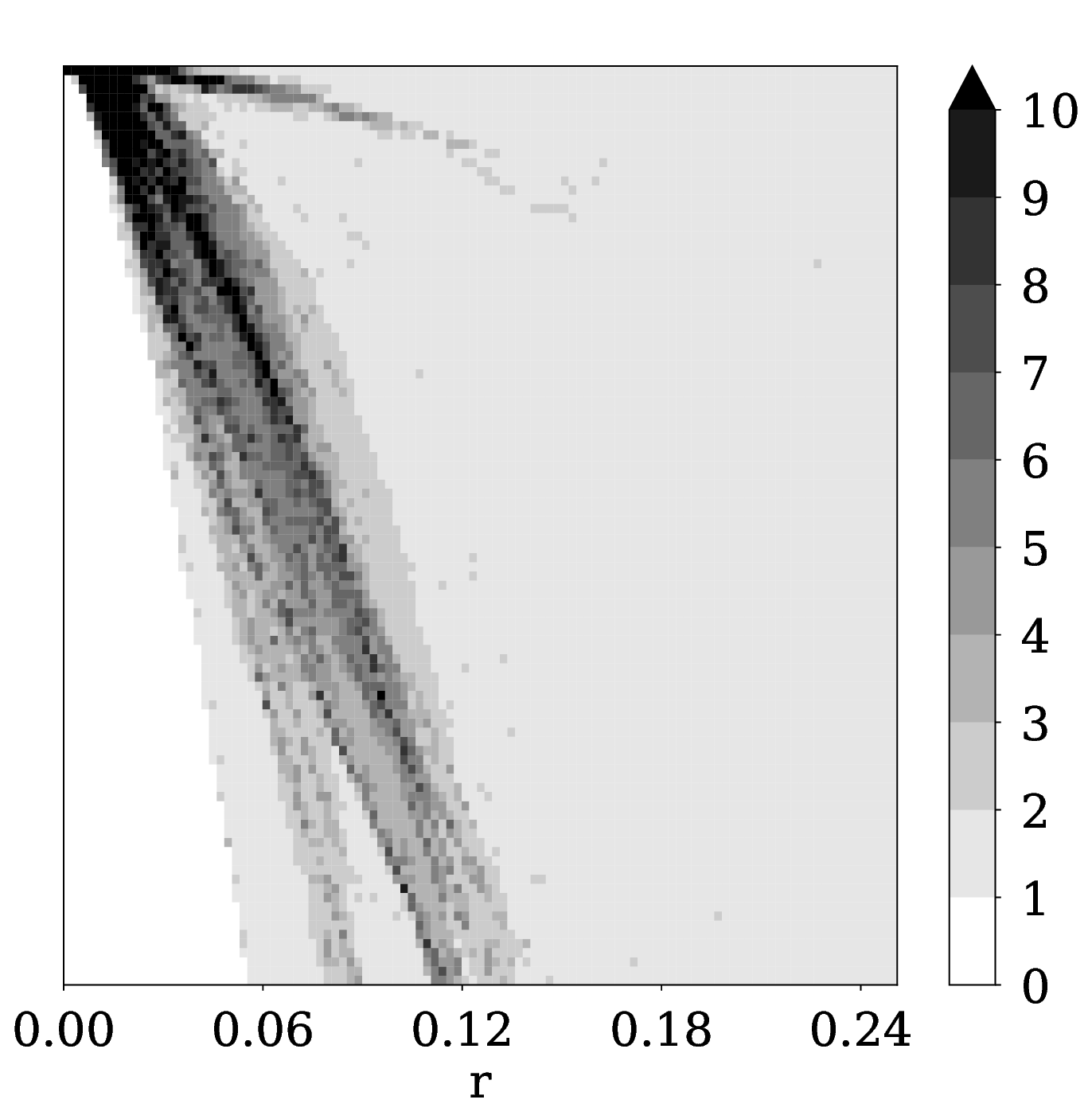}
\end{minipage}
\caption{The clustering dataset (left), together with plots visualizing the MMA (middle) and the Hilbert function (right) for the Delaunay core bifiltration in homological dimension $0$.}
\label{fig:delaunay_core_clusterable_data_point_cloud}
\end{figure}

\begin{figure}[H]
\centering
\begin{minipage}{0.33\textwidth}
\centering
\includegraphics[width=\linewidth]{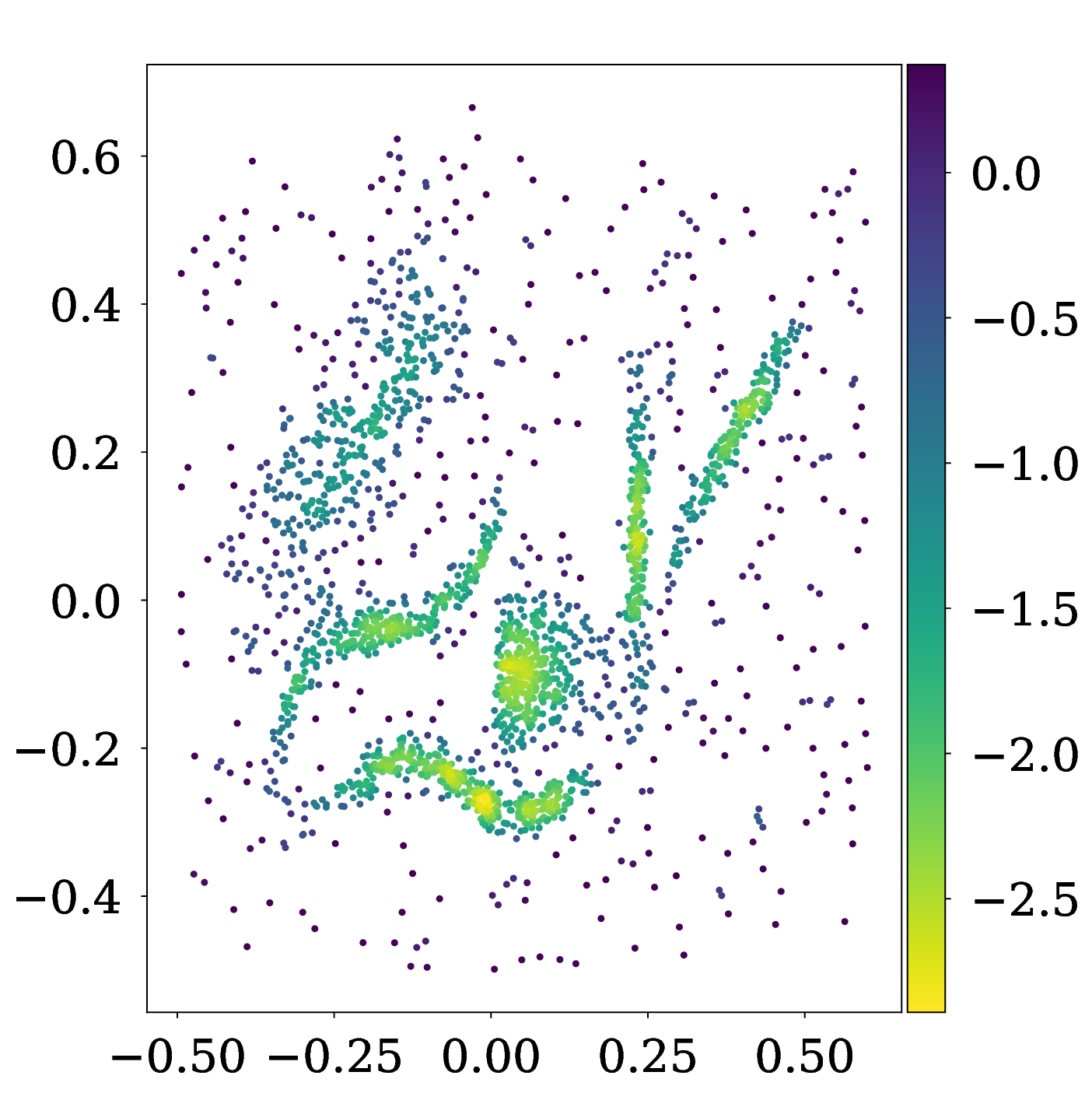}
\end{minipage}%
\begin{minipage}{0.33\textwidth}
\centering
\includegraphics[width=\linewidth]{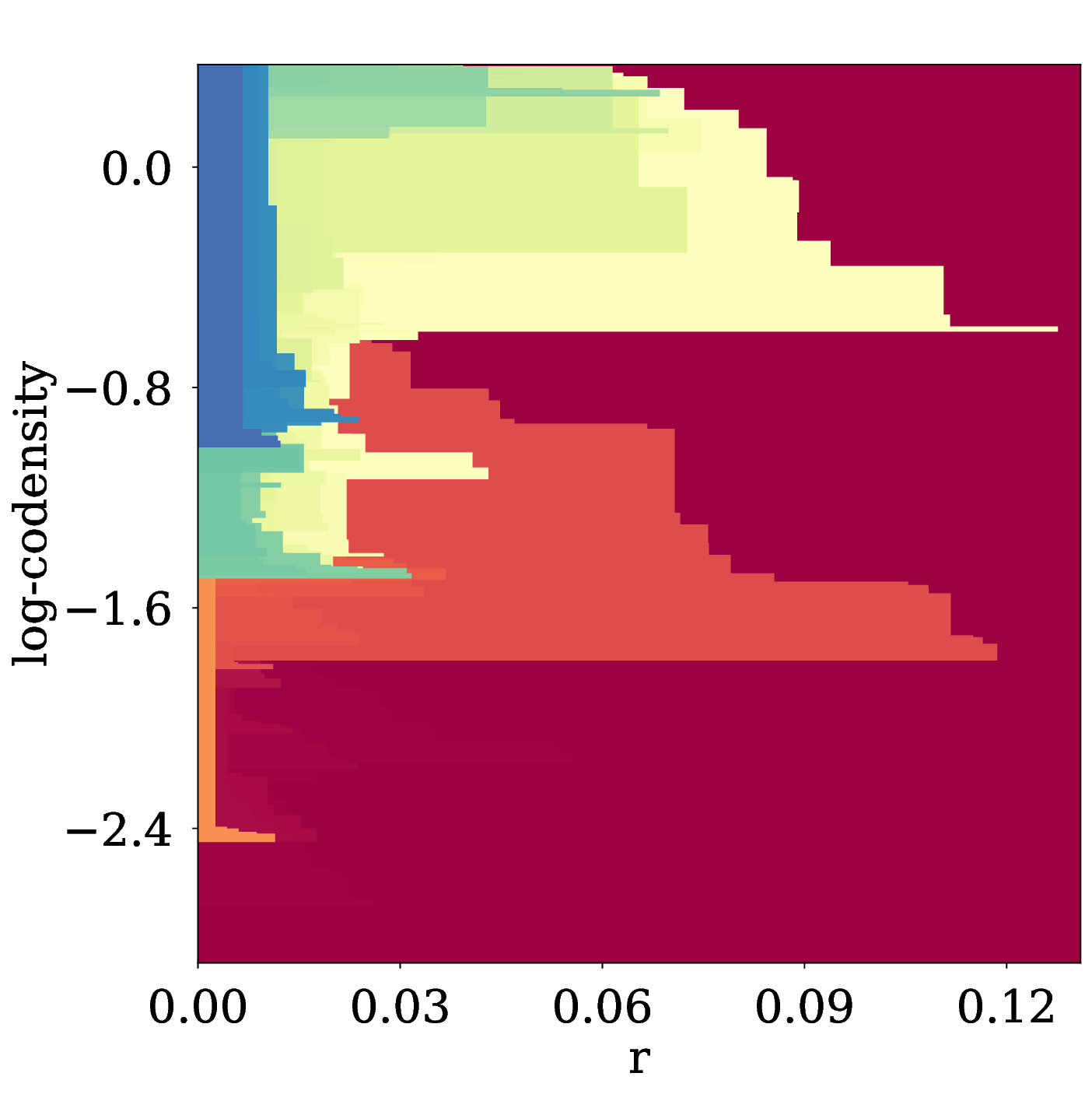}
\end{minipage}%
\begin{minipage}{0.33\textwidth}
\centering
\includegraphics[width=\linewidth]{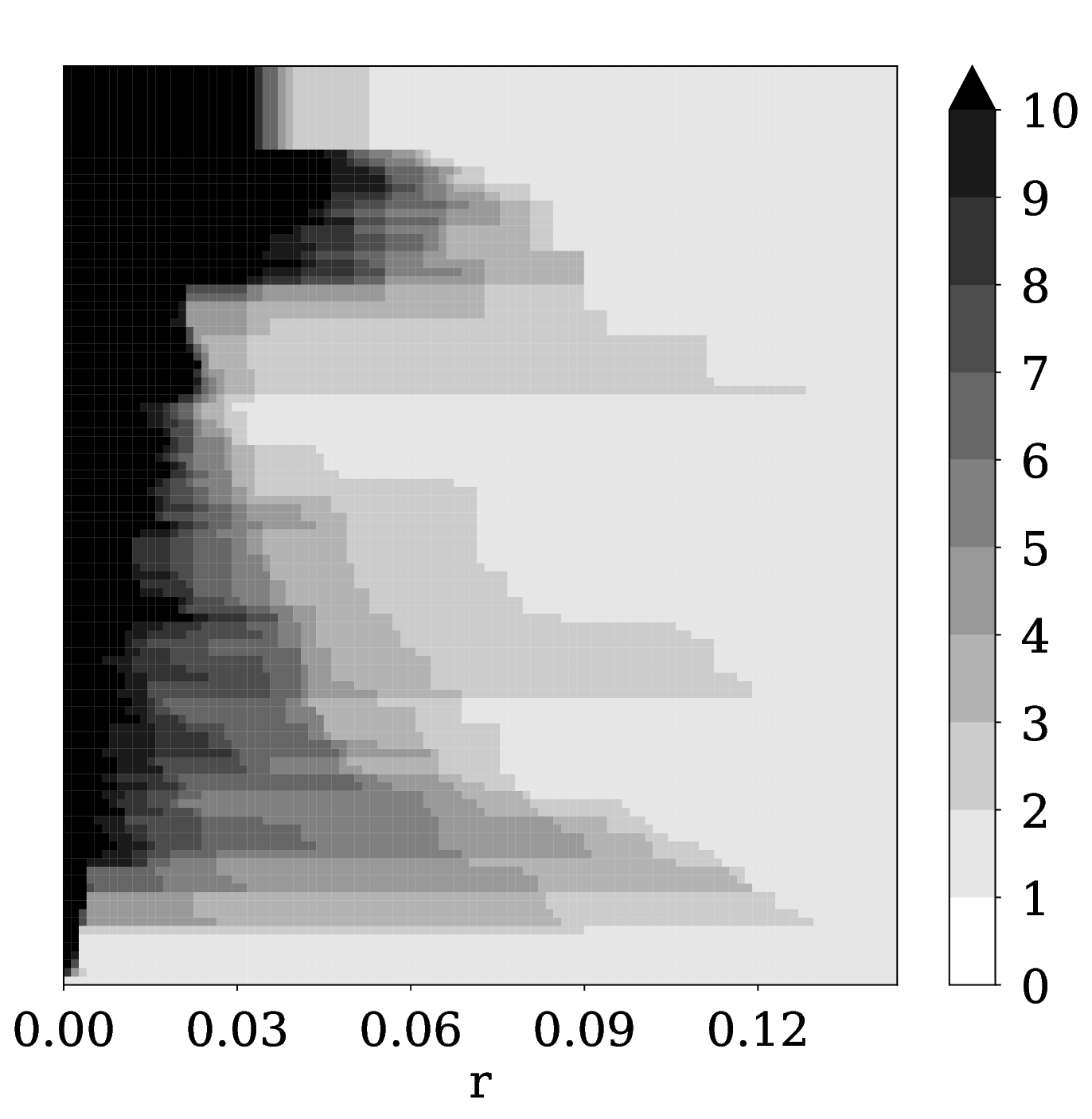}
\end{minipage}
\caption{The clustering dataset (left), colored by the codensity function $\gamma$, together with plots visualizing the MMA (middle) and the Hilbert function (right) for the sublevel Delaunay-\v{C}ech bifiltration of $\gamma$ in homological dimension $0$.}
\label{fig:function_delaunay_clusterable_data_point_cloud}
\end{figure}

\subsubsection{Comparing with the Multicover Bifiltration}\label{sec:comparing_to_multicover_rhomboid}

In this section, we compare the Delaunay core bifiltration to the multicover bifiltration, which we compute via the unsliced rhomboid tiling bifiltration from \cite{Corbet2023}. The rhomboid tiling is a polyhedral cell complex introduced in \cite{Edelsbrunner2021} and an efficient algorithm for computing it was proposed in \cite{Edelsbrunner2023}. This algorithm was extended in \cite{Corbet2023} to compute the unsliced and sliced rhomboid tiling bifiltrations, which are both weakly equivalent to the multicover bifiltration for Euclidean point clouds. Since the unsliced rhomboid tiling bifiltration is faster to compute and of smaller size than the sliced rhomboid tiling bifiltration, this is the one we use in our experiments.

For computing the unsliced rhomboid tiling bifiltration, we use the \texttt{rhomboidtiling}\footnote{\url{https://github.com/geoo89/rhomboidtiling}} software to construct its free implicit representation (FIREP). To compute Hilbert functions for the rhomboid tiling bifiltration, we use the \texttt{multipers} library, which first computes the minimal presentation, using \texttt{mpfree}\footnote{\url{https://bitbucket.org/mkerber/mpfree}} as the backend, to reduce the size of the bifiltration.

To visually compare the two bifiltrations, we perform an experiment similar to \cite{Corbet2023}. We compute the Hilbert function for $H_1$ for both the Delaunay core and the unsliced rhomboid tiling bifiltration on noisy samples from $S^1 \subseteq \mathbb{R}^2$ using $k = 1,2,\ldots,100$. Each point cloud contains $100$ points, with $n_\text{signal}$ points uniformly sampled from $S^1$ and perturbed with Gaussian noise, and $n_\text{noise}$ outliers uniformly sampled from $[-1,1]^2$. We consider three point clouds with $n_\text{noise} \in \{0,30,70\}$. The point clouds and their corresponding Hilbert functions are shown in \cref{fig:rhomboid_tiling_vs_delaunay_core_hilbert_function}.

\begin{figure}[H]
\centering
\begin{minipage}{0.32\textwidth}
\centering
\includegraphics[width=\linewidth]{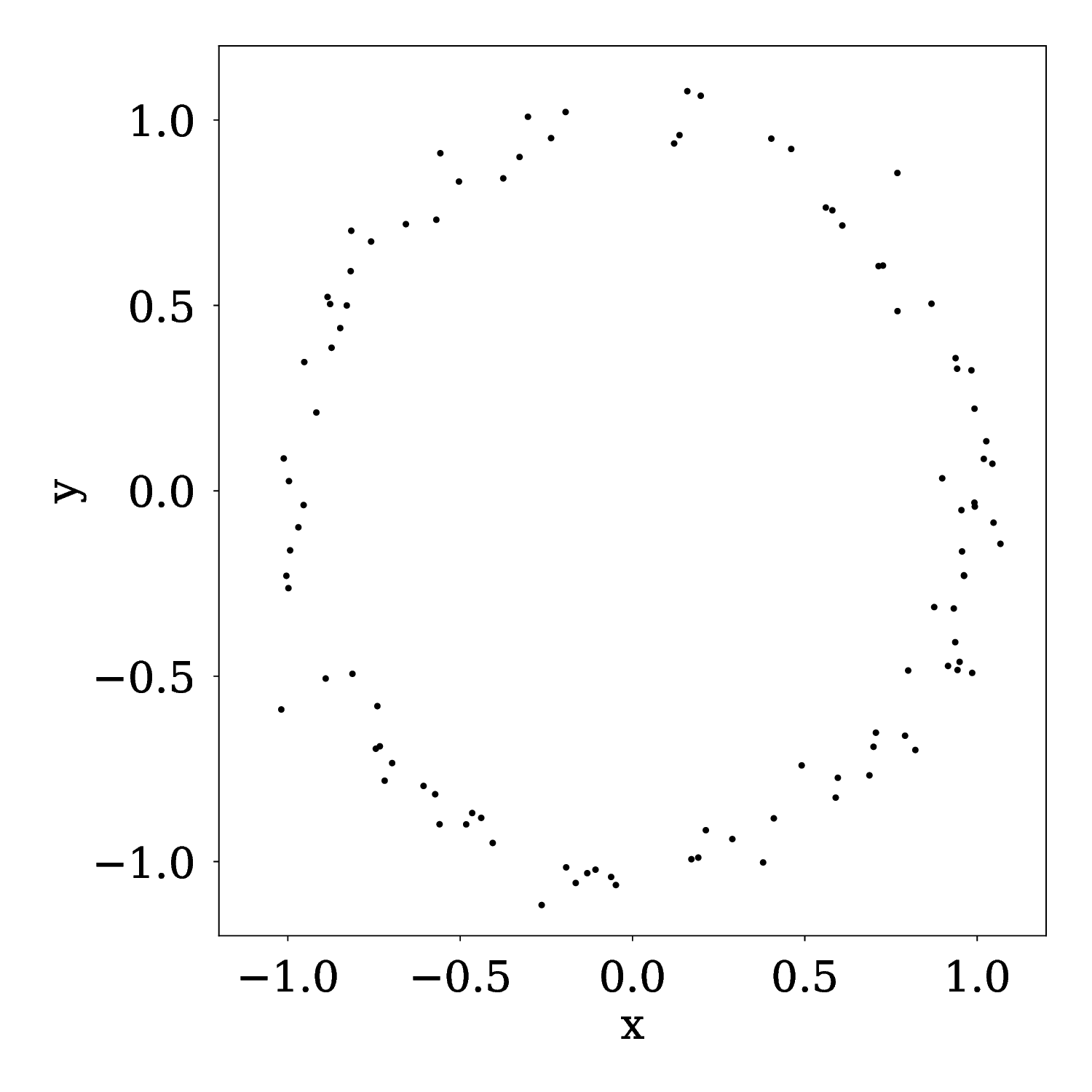}
\end{minipage}%
\begin{minipage}{0.32\textwidth}
\centering
\includegraphics[width=\linewidth]{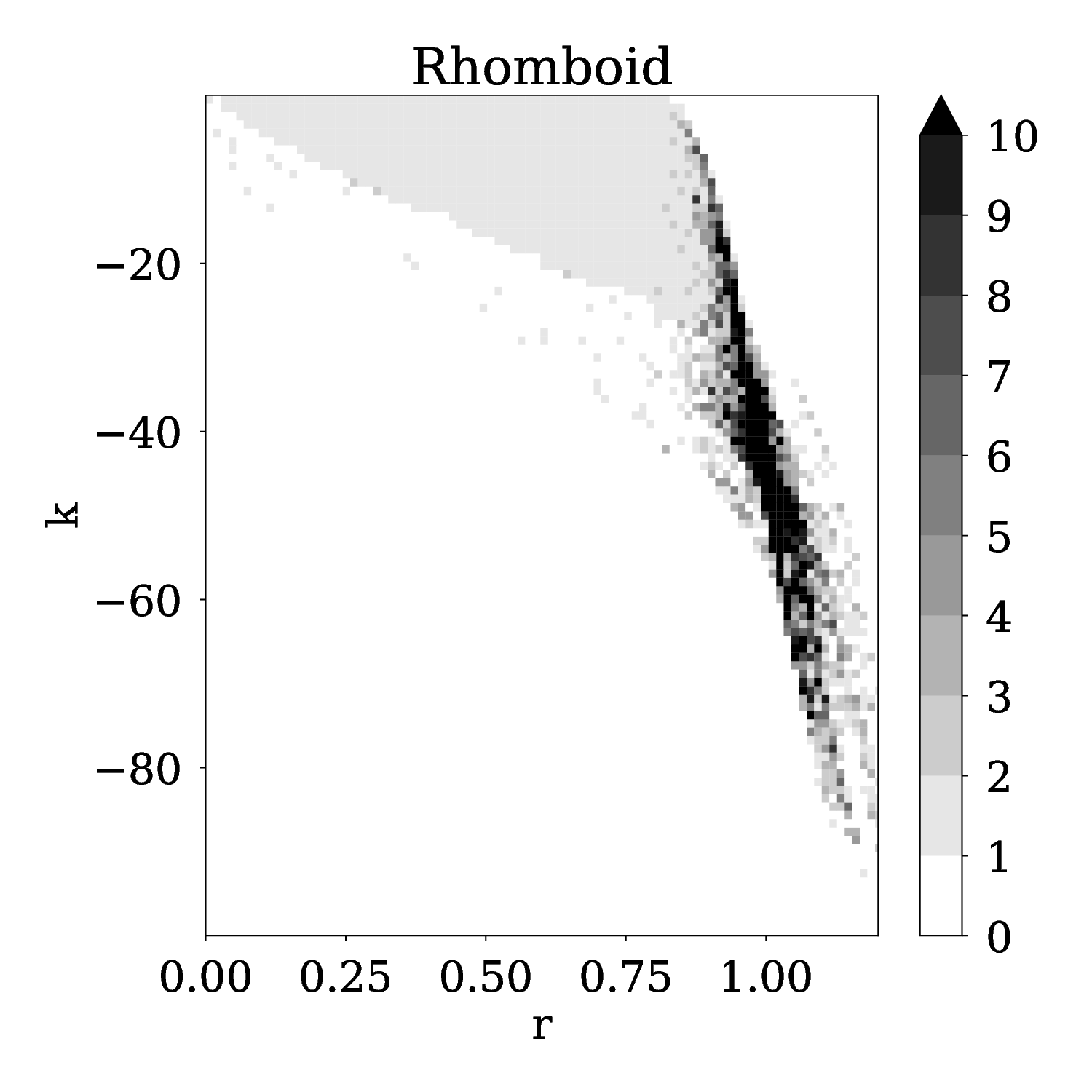}
\end{minipage}%
\begin{minipage}{0.32\textwidth}
\centering
\includegraphics[width=\linewidth]{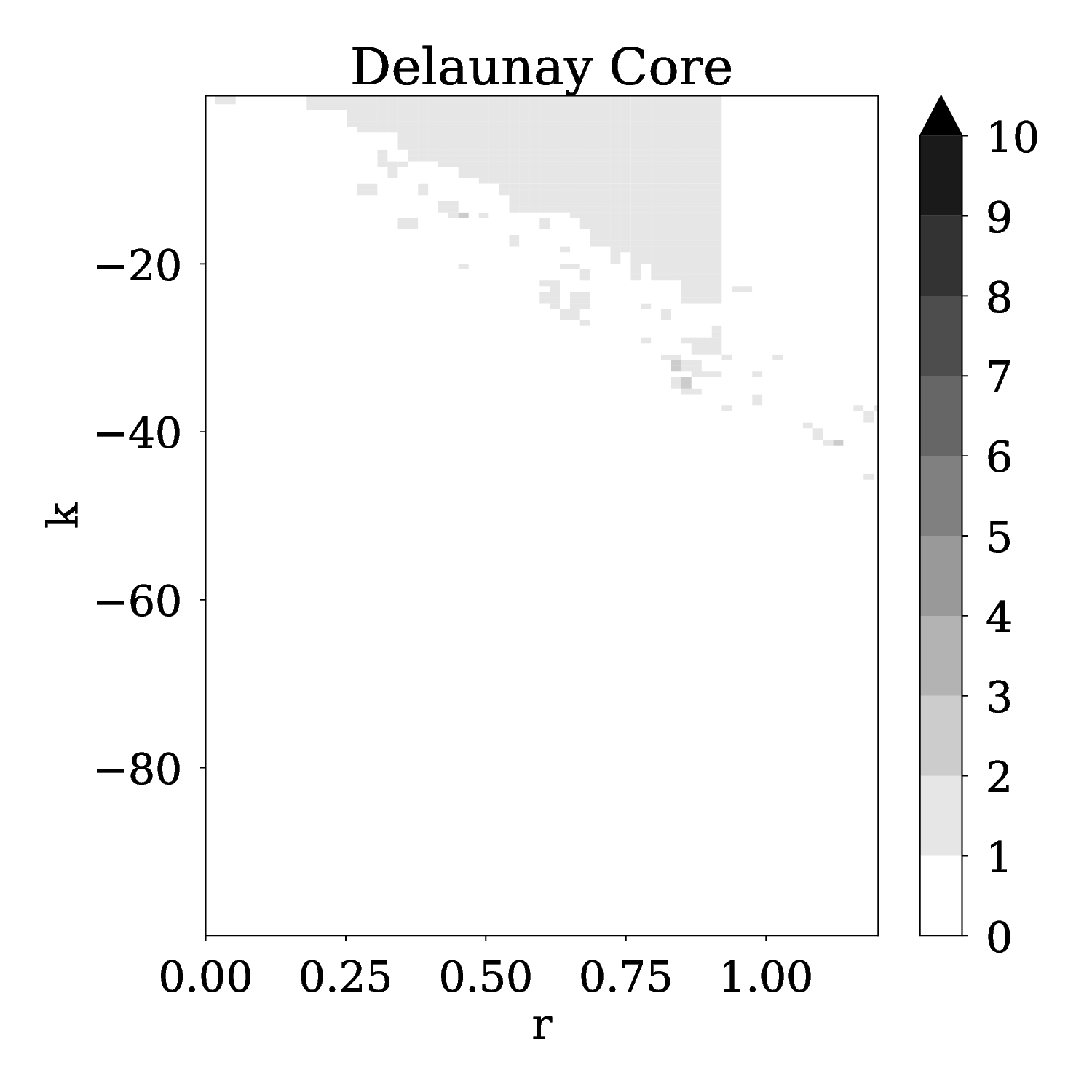}
\end{minipage}
\vspace{0.2em}
\begin{minipage}{0.32\textwidth}
\centering
\includegraphics[width=\linewidth]{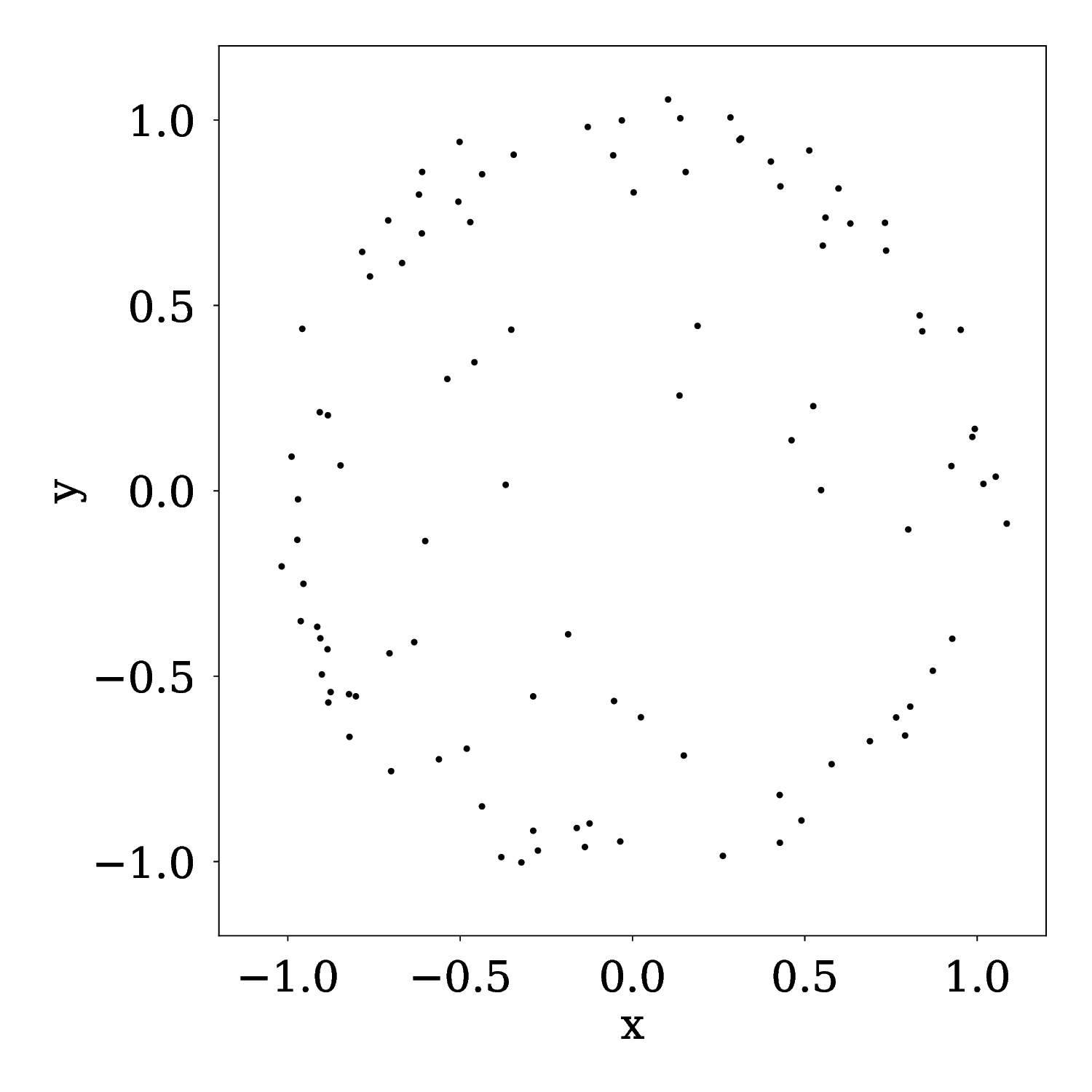}
\end{minipage}%
\begin{minipage}{0.32\textwidth}
\centering
\includegraphics[width=\linewidth]{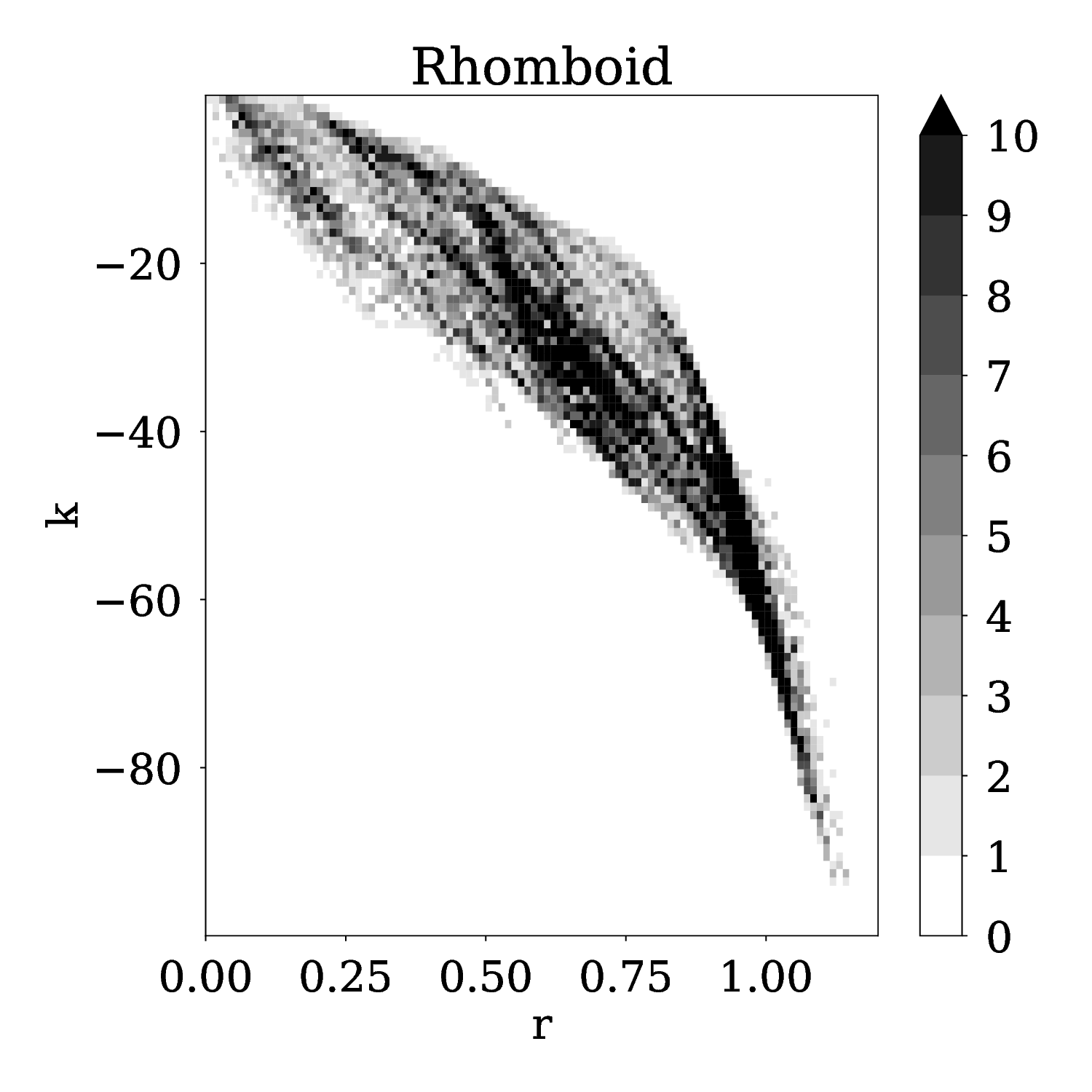}
\end{minipage}%
\begin{minipage}{0.32\textwidth}
\centering
\includegraphics[width=\linewidth]{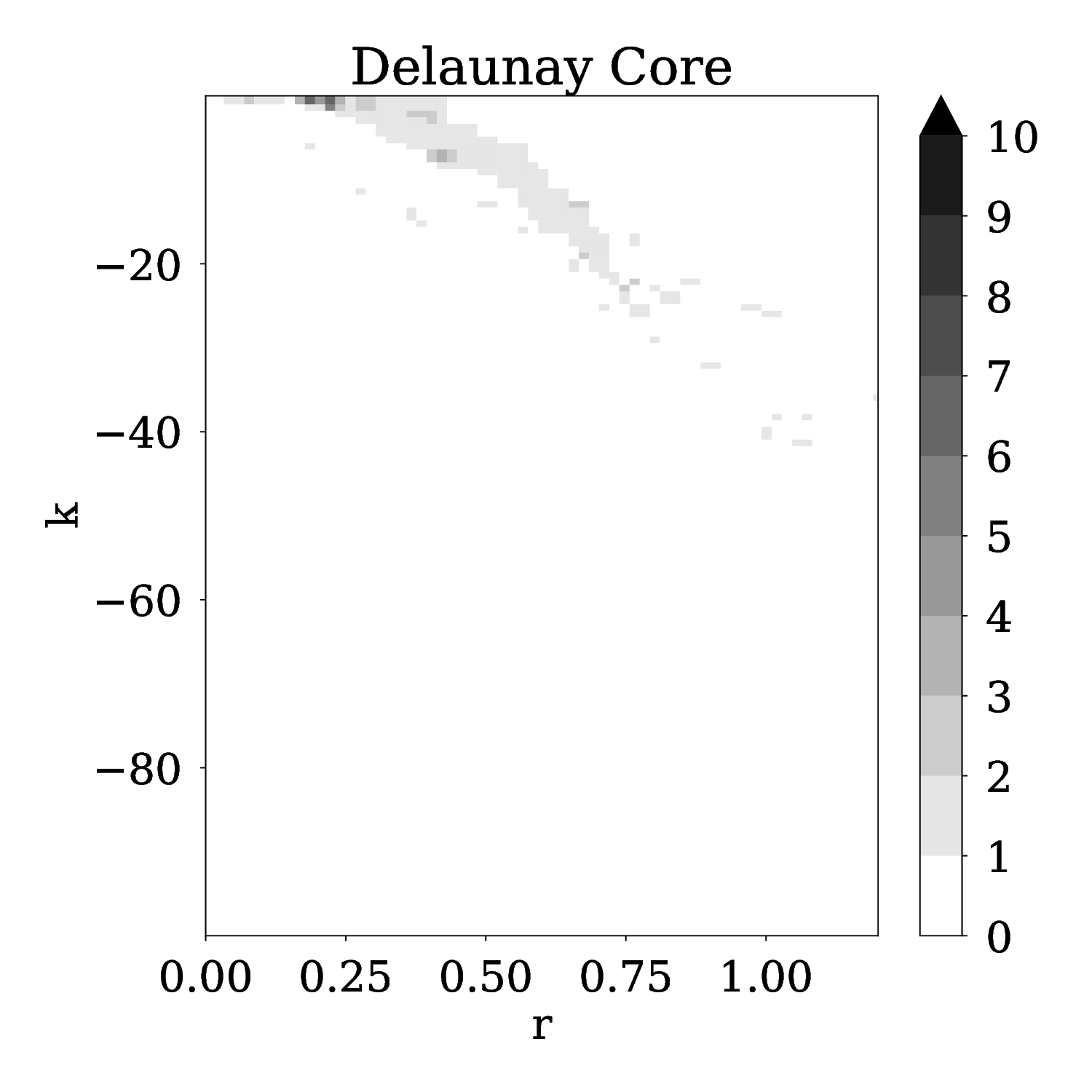}
\end{minipage}
\vspace{0.2em}
\begin{minipage}{0.32\textwidth}
\centering
\includegraphics[width=\linewidth]{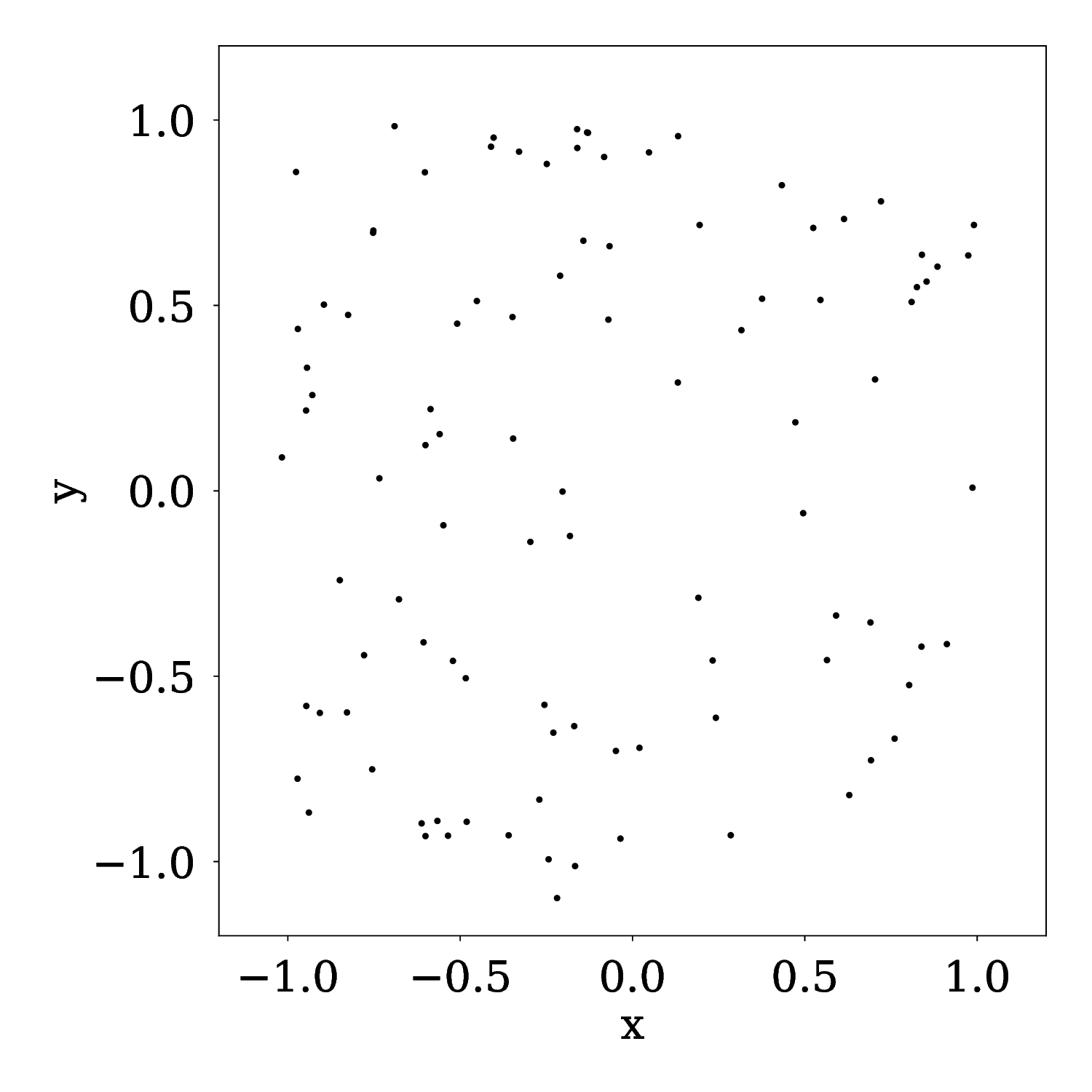}
\end{minipage}%
\begin{minipage}{0.32\textwidth}
\centering
\includegraphics[width=\linewidth]{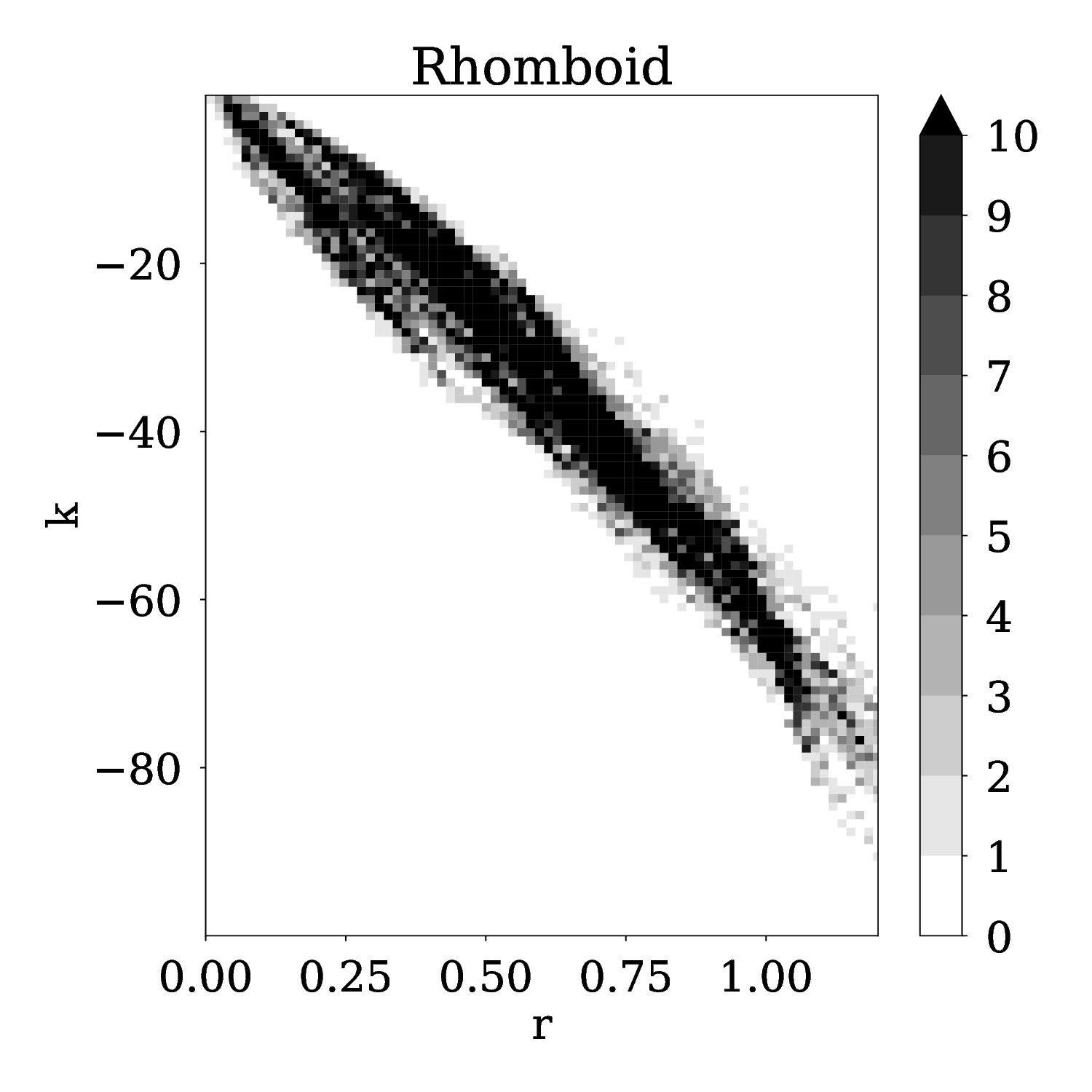}
\end{minipage}%
\begin{minipage}{0.32\textwidth}
\centering
\includegraphics[width=\linewidth]{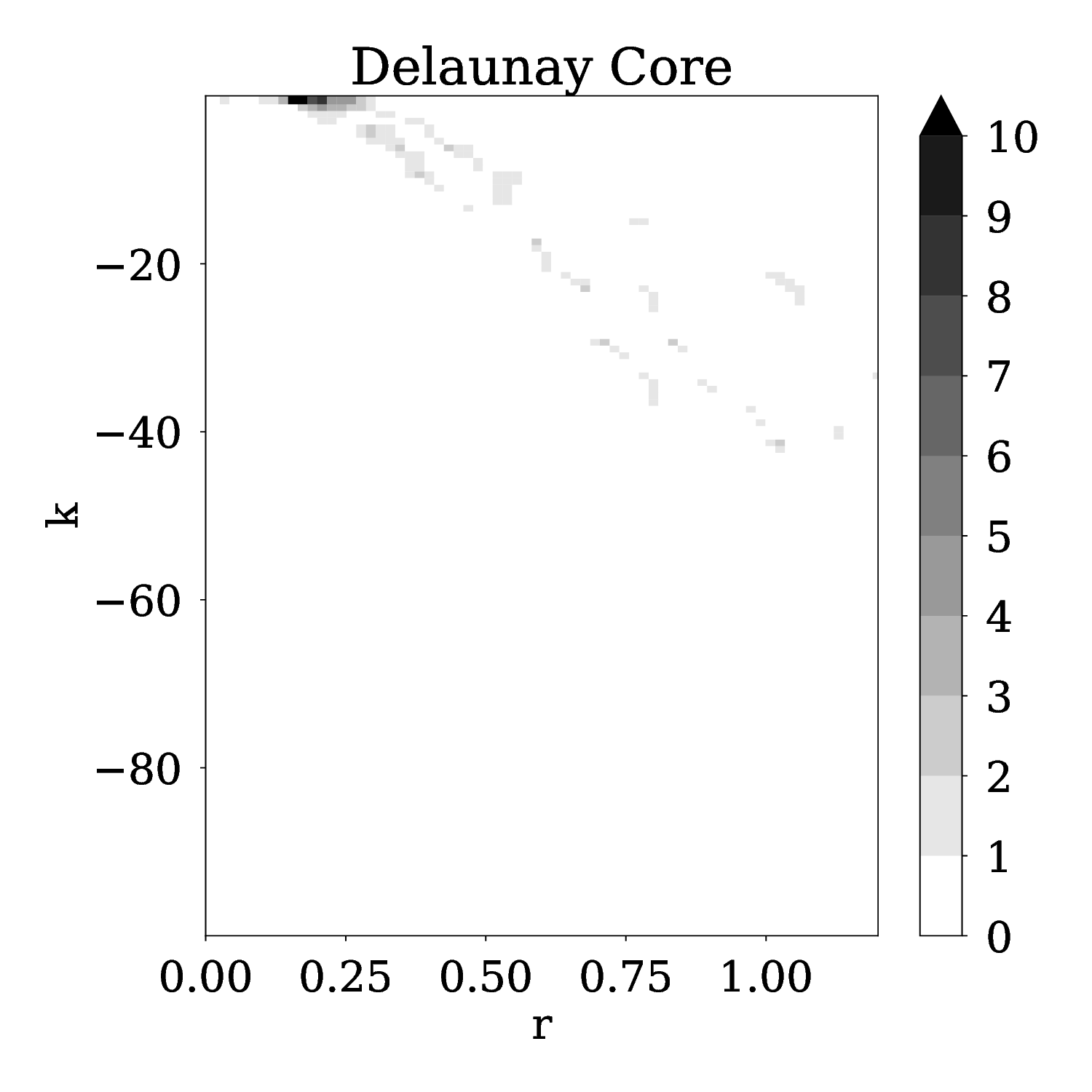}
\end{minipage}
\caption{Noisy point cloud samples from $S^1$ (left), together with the first Hilbert functions for the unsliced rhomboid tiling bifiltration (middle) and the Delaunay core bifiltration (right). From top to bottom, $n_\text{noise} = 0, 30, 70$, respectively.}
\label{fig:rhomboid_tiling_vs_delaunay_core_hilbert_function}
\end{figure}

To compare the sizes of the Delaunay core and unsliced rhomboid tiling bifiltrations, we perform benchmarks similar to \cite{Corbet2023}. We sample point clouds of size $n$ uniformly from the unit square in $\mathbb{R}^2$ for $k = 1,2,\ldots,k_\text{max}$, with $k_\text{max} \in \{4,8\}$. The resulting filtration sizes are reported in \cref{table:rhomboid_tiling_delaunay_size_benchmark}. In \cref{table:delaunay_core_bifiltration_size_and_time_benchmark} in the appendix, we list additional filtration sizes, as well as computation times for the Delaunay core bifiltration on uniform noise in both $\mathbb{R}^2$ and $\mathbb{R}^3$ using more values for $k$.

\begin{table}[H]
\centering
\begin{tabular}{ccrr}
\toprule
\multicolumn{2}{c}{} & \multicolumn{2}{c}{\textbf{Filtration size}} \\ 
\cmidrule(r){3-4}
$\mathbf{n}$ & $\mathbf{k_{\text{max}}}$ & \textbf{Delaunay core} & \textbf{Rhomboid tiling} \\
\midrule
10,000 & 4 & 210,711 & 1,197,090 \\
20,000 & 4 & 419,921 & 2,397,221 \\
40,000 & 4 & 840,774 & 4,797,268 \\
80,000 & 4 & 1,687,077 & 9,597,102 \\
\midrule
10,000 & 8 & 447,582 & 4,343,045 \\
20,000 & 8 & 894,760 & 8,702,590 \\
40,000 & 8 & 1,791,812 & 17,422,488 \\
80,000 & 8 & 3,591,491 & 34,861,283 \\
\bottomrule
\end{tabular}
\caption{Filtration sizes for the Delaunay core and unsliced rhomboid tiling bifiltrations for point clouds of size $n$ uniformly sampled from the unit square in $\mathbb{R}^2$, with different values of $k_\text{max}$.}
\label{table:rhomboid_tiling_delaunay_size_benchmark}
\end{table}

\section{Conclusions}\label{sec:conclusion}

We have introduced the core and Delaunay core bifiltrations and shown how they relate to the multicover bifiltration via interleavings. In addition, we have shown how the stability results for the multicover bifiltration directly extends to stability for the core and Delaunay core bifiltrations. For the core bifiltration, we additionally established a stability result stronger than the transferred multicover stability by giving a direct proof. We have also presented experimental results showing how Delaunay core persistence can be useful for analyzing relatively large point clouds in cases where robustness to noise is desired.

In future work, it would be interesting to see if these stability results can be strengthened further. It would also be interesting to consider stability with respect to metrics besides the Prohorov distance, and conditions on the point cloud giving stability for a fixed density parameter $k$. It is probably worthwhile to examine a core trifiltration including the parameter $\beta$.

In the applied direction, we believe it would be worthwhile to look into how we can best choose the line to compute persistence along, finding the optimal value of $\beta$, and also explore how Delaunay core persistence can be applied to real-world datasets.

\section*{Acknowledgments and Funding}
\noindent\textbf{Acknowledgments.} The authors thank David Loiseaux for assistance with the Delaunay core implementation in the \texttt{multipers} library.

\medskip

\noindent\textbf{Funding.} Lars M.\ Salbu: \sloppy This project was partially co-funded by the European Union, GA\#101126560; Bergen research and training program for future AI leaders across the disciplines, LEAD AI.

\newpage

\appendix
\section{Supporting Material}\label{sec:distance_tables}
\begin{table}[h!]
\centering
\begin{tabular}{rcccccccc} 
& \multicolumn{4}{c}{$\textbf{s}\text{ (fixed)}$} & \multicolumn{4}{c}{$\textbf{s}_{\text{max}}\text{ (}s=g(r)\text{)}$} \\
\cmidrule(lr){2-5}
\cmidrule(lr){6-9}
& \textbf{0}    & \textbf{0.001}   & \textbf{0.01}  &  \textbf{0.1}  & \textbf{0}    & \textbf{0.001}   & \textbf{0.01}  & \textbf{0.1} \\

\rowcolor{gray!10} \multicolumn{9}{l}{\textbf{Torus 1}} \\
Dim 0 & 
\textbf{0.109} & 0.161 & 0.603 & 1.392 &  
\textbf{0.109} & 0.161 & 0.614 & 1.475 \\ 

Dim 1 &
0.977 & 0.977 & 0.977 & 0.977 &  
0.977 & 0.977 & 0.977 & 0.977 \\ 

Dim 2 &
0.425 & 0.425 & 0.425 & 0.425 &  
0.425 & 0.425 & 0.425 & 0.425 \\ 

\rowcolor{gray!10} \multicolumn{9}{l}{\textbf{Torus 2}} \\
Dim 0 &
\textbf{0.086} & 0.131 & 0.388 & 0.901 &  
\textbf{0.086} & 0.131 & 0.404 & 0.958 \\ 

Dim 1 &
0.487 & 0.487 & \textbf{0.421} & 0.487 &  
0.487 & 0.487 & \textbf{0.414} & 0.487 \\ 

Dim 2 &
0.445 & 0.445 & \textbf{0.390} & 0.445 &  
0.445 & 0.445 & \textbf{0.372} & 0.445 \\ 

Dim 3 &
0.207 & 0.207 & 0.207 & 0.207 &  
0.207 & 0.207 & 0.207 & 0.207 \\ 

\rowcolor{gray!10} \multicolumn{9}{l}{\textbf{Sphere}} \\
Dim 0 &
\textbf{0.042} & 0.073 & 0.220 & 0.595 &  
\textbf{0.042} & 0.073 & 0.227 & 0.635 \\ 

Dim 1 &
0.029 & \textbf{0.016} & 0.016 & 0.016 &  
0.029 & \textbf{0.016} & 0.016 & 0.016 \\ 

Dim 2 &
0.475 & 0.475 & 0.475 & 0.475 &  
0.475 & 0.475 & 0.475 & 0.475 \\ 

\rowcolor{gray!10} \multicolumn{9}{l}{\textbf{Circle}} \\
Dim 0 &
\textbf{0.012} & 0.014 & 0.073 & 0.342 &  
\textbf{0.012} & 0.014 & 0.074 & 0.371 \\ 

Dim 1 &
0.499 & 0.499 & 0.499 & \textbf{0.432} &  
0.499 & 0.499 & 0.499 & \textbf{0.393} \\ 

\rowcolor{gray!10} \multicolumn{9}{l}{\textbf{Circles}} \\
Dim 0 &
0.125 & 0.125 & \textbf{0.096} & 0.357 &  
0.125 & 0.125 & \textbf{0.092} & 0.377 \\ 

Dim 1 &
0.249 & 0.249 & 0.249 & 0.249 &  
0.249 & 0.249 & 0.249 & 0.249 \\ 

\end{tabular}
\caption{Bottleneck distances between the ground truth and the Delaunay core persistence diagrams where $n=10000$ and $m=10000$ ($1:1$ signal-to-noise ratio). Distances are shown for both $s$ fixed and for persistence along the line $s=g(r)$ determined by $s_\text{max}$, $n+m$ and the diameter of the input point cloud.}
\label{table:bottleneck_distances_n_10000_m_10000}
\end{table}

\begin{table}[h!]
\centering
\begin{tabular}{rcccccccc} 
& \multicolumn{4}{c}{$\textbf{s}\text{ (fixed)}$} & \multicolumn{4}{c}{$\textbf{s}_{\text{max}}\text{ (}s=g(r)\text{)}$} \\
\cmidrule(lr){2-5}
\cmidrule(lr){6-9}
& \textbf{0}    & \textbf{0.001}   & \textbf{0.01}  &  \textbf{0.1}  & \textbf{0}    & \textbf{0.001}   & \textbf{0.01}  & \textbf{0.1} \\

\rowcolor{gray!10} \multicolumn{9}{l}{\textbf{Torus 1}} \\
Dim 0 &
0.187 & \textbf{0.114} & 0.505 & 1.302 &  
0.187 & \textbf{0.114} & 0.518 & 1.381 \\ 

Dim 1 &
0.968 & 0.968 & \textbf{0.545} & 0.968 &  
0.968 & 0.968 & \textbf{0.554} & 0.968 \\ 

Dim 2 &
0.270 & \textbf{0.183} & 0.387 & 0.388 &  
0.270 & \textbf{0.183} & 0.388 & 0.388 \\ 

\rowcolor{gray!10} \multicolumn{9}{l}{\textbf{Torus 2}} \\
Dim 0 &
0.147 & \textbf{0.089} & 0.315 & 0.895 &  
0.147 & \textbf{0.089} & 0.324 & 0.973 \\ 

Dim 1 &
0.482 & 0.482 & \textbf{0.343} & 0.482 &  
0.482 & 0.482 & \textbf{0.338} & 0.482 \\ 

Dim 2 &
0.441 & 0.420 & \textbf{0.288} & 0.441 &  
0.441 & 0.411 & \textbf{0.303} & 0.441 \\ 

Dim 3 &
0.207 & 0.207 & 0.207 & \textbf{0.178} &  
0.207 & 0.207 & 0.207 & \textbf{0.154} \\ 

\rowcolor{gray!10} \multicolumn{9}{l}{\textbf{Sphere}} \\
Dim 0 &
0.091 & \textbf{0.055} & 0.186 & 0.548 &  
0.091 & \textbf{0.055} & 0.190 & 0.589 \\ 

Dim 1 &
0.044 & \textbf{0.020} & 0.020 & 0.020 &  
0.044 & \textbf{0.020} & 0.020 & 0.020 \\ 

Dim 2 &
0.470 & 0.470 & \textbf{0.340} & 0.470 &  
0.470 & 0.470 & \textbf{0.302} & 0.470 \\ 

\rowcolor{gray!10} \multicolumn{9}{l}{\textbf{Circle}} \\
Dim 0 &
0.034 & \textbf{0.020} & 0.053 & 0.288 &  
0.034 & \textbf{0.020} & 0.053 & 0.312 \\ 

Dim 1 &
0.499 & 0.499 & 0.499 & \textbf{0.328} &  
0.499 & 0.499 & 0.499 & \textbf{0.357} \\ 

\rowcolor{gray!10} \multicolumn{9}{l}{\textbf{Circles}} \\
Dim 0 &
0.125 & 0.114 & \textbf{0.092} & 0.340 &  
0.125 & 0.114 & \textbf{0.088} & 0.356 \\ 

Dim 1 &
0.248 & 0.248 & \textbf{0.162} & 0.248 &  
0.248 & 0.248 & \textbf{0.159} & 0.248 \\ 

\end{tabular}
\caption{Bottleneck distances between the ground truth and the Delaunay core persistence diagrams where $n=10000$ and $m=1000$ ($10:1$ signal-to-noise ratio). Distances are shown for both $s$ fixed and for persistence along the line $s=g(r)$ determined by $s_\text{max}$, $n+m$ and the diameter of the input point cloud.}
\label{table:bottleneck_distances_n_10000_m_1000}
\end{table}

\begin{table}[h!]
\centering
\begin{tabular}{rcccccccc} 
& \multicolumn{4}{c}{$\textbf{s}\text{ (fixed)}$} & \multicolumn{4}{c}{$\textbf{s}_{\text{max}}\text{ (}s=g(r)\text{)}$} \\
\cmidrule(lr){2-5}
\cmidrule(lr){6-9}
& \textbf{0}    & \textbf{0.001}   & \textbf{0.01}  &  \textbf{0.1}  & \textbf{0}    & \textbf{0.001}   & \textbf{0.01}  & \textbf{0.1} \\

\rowcolor{gray!10} \multicolumn{9}{l}{\textbf{Torus 1}} \\
Dim 0 &
0.273 & \textbf{0.118} & 0.501 & 1.355 &  
0.273 & \textbf{0.118} & 0.509 & 1.420 \\ 

Dim 1 &
0.967 & 0.684 & \textbf{0.549} & 0.967 &  
0.967 & 0.662 & \textbf{0.546} & 0.967 \\ 

Dim 2 &
0.136 & \textbf{0.125} & 0.391 & 0.391 &  
0.136 & \textbf{0.125} & 0.390 & 0.391 \\ 

\rowcolor{gray!10} \multicolumn{9}{l}{\textbf{Torus 2}} \\
Dim 0 &
0.178 & \textbf{0.092} & 0.300 & 0.879 &  
0.178 & \textbf{0.092} & 0.308 & 0.986 \\ 

Dim 1 &
0.478 & 0.409 & \textbf{0.387} & 0.482 &  
0.478 & 0.405 & \textbf{0.387} & 0.482 \\ 

Dim 2 &
0.330 & \textbf{0.227} & 0.277 & 0.439 &  
0.330 & \textbf{0.227} & 0.291 & 0.439 \\ 

Dim 3 &
0.207 & 0.207 & 0.207 & 0.207 &  
0.207 & 0.207 & 0.207 & 0.207 \\ 

\rowcolor{gray!10} \multicolumn{9}{l}{\textbf{Sphere}} \\
Dim 0 &
0.146 & \textbf{0.049} & 0.175 & 0.544 &  
0.146 & \textbf{0.049} & 0.177 & 0.590 \\ 

Dim 1 &
\textbf{0.020} & 0.022 & 0.022 & 0.022 &  
\textbf{0.020} & 0.022 & 0.022 & 0.022 \\ 

Dim 2 &
0.465 & 0.423 & \textbf{0.283} & 0.465 &  
0.465 & 0.423 & \textbf{0.278} & 0.465 \\ 

\rowcolor{gray!10} \multicolumn{9}{l}{\textbf{Circle}} \\
Dim 0 &
0.106 & \textbf{0.010} & 0.051 & 0.275 &  
0.106 & \textbf{0.010} & 0.052 & 0.301 \\ 

Dim 1 &
0.499 & 0.499 & \textbf{0.270} & 0.308 &  
0.499 & 0.499 & \textbf{0.263} & 0.338 \\ 

\rowcolor{gray!10} \multicolumn{9}{l}{\textbf{Circles}} \\
Dim 0 &
\textbf{0.064} & 0.125 & 0.099 & 0.341 &  
\textbf{0.064} & 0.125 & 0.098 & 0.356 \\ 

Dim 1 &
0.248 & 0.248 & \textbf{0.223} & 0.248 &  
0.248 & 0.248 & \textbf{0.219} & 0.248 \\ 

\end{tabular}
\caption{Bottleneck distances between the ground truth and the Delaunay core persistence diagrams where $n=10000$ and $m=100$ ($100:1$ signal-to-noise ratio). Distances are shown for both $s$ fixed and for persistence along the line $s=g(r)$ determined by $s_\text{max}$, $n+m$ and the diameter of the input point cloud.}
\label{table:bottleneck_distances_n_10000_m_100}
\end{table}

\begin{table}[h!]
\centering
\begin{tabular}{r*{4}{c}}
\toprule
& \multicolumn{4}{c}{\textbf{Fixed} $\mathbf{k}$} \\
\cmidrule(lr){2-5}
$\mathbf{n+m}$ & 10 & 100 & 1000 & 10000 \\
\cmidrule(lr){1-1}
\cmidrule(lr){2-5}
10000 & 1.11 & 1.14 & 1.72 & 7.86 \\
20000 & 2.32 & 2.42 & 3.59 & 16.90 \\
30000 & 3.57 & 3.75 & 5.58 & 26.71 \\
40000 & 4.85 & 5.08 & 7.50 & 35.15 \\
50000 & 6.17 & 6.53 & 9.76 & 47.29 \\
60000 & 7.51 & 7.91 & 11.70 & 56.64 \\
\bottomrule
\end{tabular}
\caption{Runtimes (in seconds) for computing Delaunay core persistence on the \textbf{Torus 1} dataset with $n=m=\vert X\vert/2$ and $\sigma=0.07$. The computations were conducted on an Intel Core i5-6300U (at 2.40 GHz) CPU with 16 GB RAM. The minimum runtime out of $10$ runs is reported in the table.}
\label{table:runtimes_torus_1}
\end{table}

\begin{figure}[h]
\centering

\includegraphics[width=\textwidth]{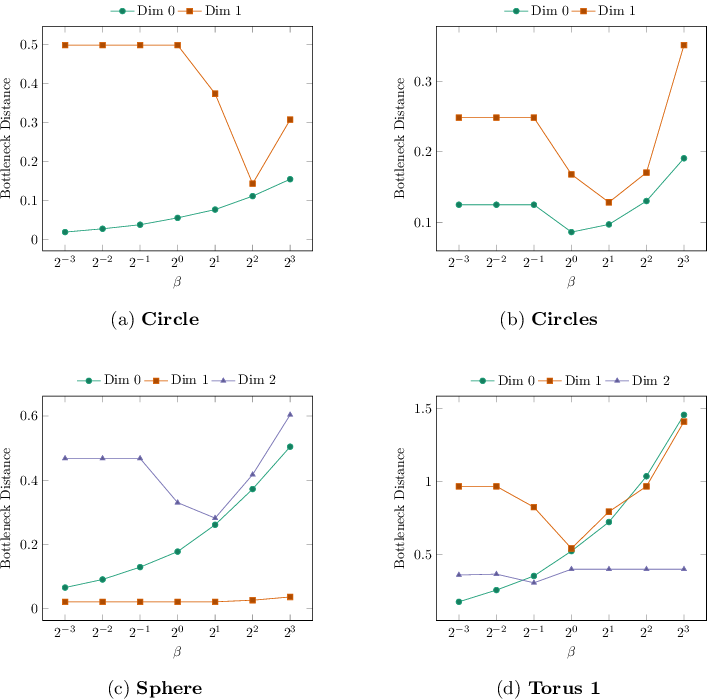}

\caption{Bottleneck distances between the ground truth persistence diagrams and the Delaunay core persistence diagrams, showing the effect of varying $\beta\in\{0.125, 0.25, 0.5, 1, 2, 4, 8\}$. All point clouds were sampled with $n=10000$ (signal), $m=1000$ (noise) and $\sigma=0.07$. Delaunay core persistence was computed along the line determined by $s_{\text{max}}=0.01$ (corresponding to $k_{\text{max}}=110$ in this case).}
\label{fig:beta_bottleneck_distance_plots}
\end{figure}

\begin{algorithm}[H]
\caption{Algorithm for computing the Delaunay core bifiltration}\label{alg:delaunay_core}
\begin{algorithmic}
\Require Finite point cloud $A \subseteq \mathbb{R}^d$, parameter $\beta > 0$, list of $k$-values $K_{\text{list}}$
\Ensure A function $F$ assigning bifiltration values to simplices in $\operatorname{Alpha}(A)$

\State Construct the alpha complex $\operatorname{Alpha}(A)$ with filtration values $\alpha(\sigma)$

\State Initialize $F$ as an empty map $F\colon\operatorname{Alpha}(A)\to\mathcal{P}((0,\infty)\times(0,\infty)^{op})$

\For{each simplex $\sigma \in \operatorname{Alpha}(A)$}
    \State Use $\alpha(\sigma)$, the precomputed filtration value
    \State Initialize $F(\sigma)=\emptyset$
    \For{each $k \in K_{\text{list}}$}
        \State Compute the maximum $D_k(\sigma)$ of $d_k^A(a)$ for $a\in\sigma$
        \State Add the filtration value $(\max(\alpha(\sigma), \beta D_k(\sigma)), k)$ to $F(\sigma)$
    \EndFor
\EndFor
\State \Return $F$
\end{algorithmic}
\caption{Construction of the full Delaunay core bifiltration of a point cloud in Euclidean space, starting from the alpha complex. Non-critical filtration values are computed but can be omitted. In the \texttt{multipers} implementation of the Delaunay core bifiltration, they are automatically discarded in the multiparameter simplex tree, and are thus not included in further computations.}
\end{algorithm}

\begin{table}[ht]
\centering
\begin{tabular}{ccccrr}
\toprule
$\mathbf{n}$ & $\mathbf{d}$ & $\mathbf{k_{\text{max}}}$ & $\mathbf{k_{\text{step}}}$ & \textbf{Filtration size} & \textbf{Time (seconds)} \\
\midrule
10,000 & 2 & 100 & 1 & 5,932,251 & 3.13 \\
20,000 & 2 & 100 & 1 & 11,888,078 & 6.30 \\
30,000 & 2 & 100 & 1 & 17,848,139 & 9.68 \\
40,000 & 2 & 100 & 1 & 23,803,095 & 13.19 \\
\midrule
10,000 & 2 & 1000 & 10 & 5,972,335 & 4.12 \\
20,000 & 2 & 1000 & 10 & 11,963,579 & 8.33 \\
30,000 & 2 & 1000 & 10 & 17,950,931 & 12.68 \\
40,000 & 2 & 1000 & 10 & 23,947,714 & 16.99 \\
\midrule
10,000 & 3 & 100 & 1 & 27,287,112 & 18.10 \\
20,000 & 3 & 100 & 1 & 55,139,407 & 37.10 \\
30,000 & 3 & 100 & 1 & 83,328,618 & 55.93 \\
40,000 & 3 & 100 & 1 & 111,429,716 & 75.10 \\
\midrule
10,000 & 3 & 1000 & 10 & 27,879,586 & 19.47 \\
20,000 & 3 & 1000 & 10 & 56,363,053 & 39.08 \\
30,000 & 3 & 1000 & 10 & 85,007,731 & 59.07 \\
40,000 & 3 & 1000 & 10 & 113,475,183 & 79.28 \\
\bottomrule
\end{tabular}
\caption{Filtration sizes and corresponding computation times for the Delaunay core bifiltration with $k\in\{1,1+k_\text{step},1+2k_\text{step},\ldots,k_\text{max}\}$ on point clouds consisting of $n$ points uniformly sampled from $[0,1]^d\subseteq\mathbb{R}^d$.}
\label{table:delaunay_core_bifiltration_size_and_time_benchmark}
\end{table}

\clearpage

\bibliography{bibliography}

\end{document}